\documentclass[letterpaper, 11pt,  reqno]{amsart}

\usepackage{amsmath,amssymb,amscd,amsthm,amsxtra, esint, mathtools, xcolor}
\usepackage{latexsym}
\usepackage{skak}
\usepackage{accents, cancel}
\usepackage{bbm}
\usepackage{enumerate}

\usepackage[implicit=true]{hyperref}

\usepackage{stmaryrd}
\usepackage{xsavebox}

\usepackage{tikz}
\usetikzlibrary{calc,intersections,through,backgrounds}
\usetikzlibrary{decorations.pathreplacing,decorations.pathmorphing,decorations.markings,decorations.shapes}
\usetikzlibrary{shapes}

\newsavebox{\sticka}
\savebox{\sticka}[4pt]{
\begin{tikzpicture}[x=1pt, y=1pt]
\draw[line width= 1] (0,0) -- (0,6); 
\path [draw, fill] (0,7.5) circle (1.5pt); 
\end{tikzpicture}
}

\allowdisplaybreaks[2]

\newtheorem{theorem}{Theorem} [section]

\newtheorem{lemma}[theorem]{Lemma}
\newtheorem{proposition}[theorem]{Proposition}

\theoremstyle{definition}
\newtheorem{remark}[theorem]{Remark}

\DeclareMathOperator*{\iintt}{\iint}

\newcommand{\noi}{\noindent}
\newcommand{\Z}{\mathbb{Z}}
\newcommand{\R}{\mathbb{R}}
\newcommand{\C}{\mathbb{C}}
\newcommand{\T}{\mathbb{T}}

\newcommand{\wick}[1]{\hspace{-2pt}:\hspace{-2pt}{#1}\hspace{-2pt}:\hspace{-2pt}}

\newcommand{\X}{ X^\alpha}
\renewcommand{\H}{\mathcal H}
\newcommand{\W}{\mathcal W}

\renewcommand{\u}{{\mathbf u}}
\renewcommand{\v}{{\mathbf v}}
\newcommand{\w}{{\mathbf w}}
\newcommand{\0}{\mathbf 0}

\newcommand{\norm}[1]{\left\|#1\right\|}
\renewcommand{\vec}[2]{\begin{pmatrix} #1 \\ #2 \end{pmatrix}}

\def\stick{\mathchoice
    {\raisebox{-1pt}{\usebox{\sticka}}}%
    {\raisebox{-1pt}{\usebox{\sticka}}}%
    {}%
    {}}

\newcommand{\fb}{\mathbf{f}}
\newcommand{\gb}{\mathbf{g}}

\DeclareMathOperator{\Lip}{Lip}

\newcommand{\prob}{\mathbb{P}}  
\newcommand{\Pt}[1]
{\mathcal{P}_{#1}}

\newcommand{\Q}{\mathcal{Q}}

\newcommand{\E}{\mathbb{E}}

\renewcommand{\L}{\mathcal{L}}

\newcommand{\B}{\mathcal{B}}
\newcommand{\Gb}{\mathbf{\Gamma}}

\newcommand{\F}{\mathcal{F}}

\def\norm#1{\|#1\|}

\newcommand{\al}{\alpha}
\newcommand{\be}{\beta}
\newcommand{\dl}{\delta}

\newcommand{\nb}{\nabla}

\newcommand{\ep}{\varepsilon}

\newcommand{\g}{\gamma}

\newcommand{\s}{\sigma}

\newcommand{\ft}{\widehat}

\newcommand{\wt}{\widetilde}
\newcommand{\cj}{\overline}

\newcommand{\dt}{\partial_t}

\newcommand{\embeds}{\hookrightarrow}

\newcommand{\ta}{\theta}

\renewcommand{\O}{\Omega}

\newcommand{\les}{\lesssim}

\newcommand{\wto}{\rightharpoonup}

\newcommand{\jb}[1]
{\langle #1 \rangle}

\newcommand{\ind}{\mathbf 1}

\newcommand{\N}{\mathbb{N}}

\newcommand{\NN}{\mathcal{N}}

\newcommand{\CC}{\mathcal{C}}

\newcommand{\EE}{\mathcal{E}}

\renewcommand{\H}{\mathcal{H}}

\newcommand{\Wal}{\mathcal{W}^{\al,\frac{2}{\al}}}

\newtheorem*{ackno}{Acknowledgements}

\newcommand{\eps}{\ep}

\numberwithin{equation}{section}
\numberwithin{theorem}{section}

\DeclareFontFamily{U}{matha}{\hyphenchar\font45}
\DeclareFontShape{U}{matha}{m}{n}{
      <5> <6> <7> <8> <9> <10> gen * matha
      <10.95> matha10 <12> <14.4> <17.28> <20.74> <24.88> matha12
      }{}
\DeclareSymbolFont{matha}{U}{matha}{m}{n}
\DeclareFontSubstitution{U}{matha}{m}{n}

\DeclareFontFamily{U}{mathx}{\hyphenchar\font45}
\DeclareFontShape{U}{mathx}{m}{n}{
      <5> <6> <7> <8> <9> <10>
      <10.95> <12> <14.4> <17.28> <20.74> <24.88>
      mathx10
      }{}
\DeclareSymbolFont{mathx}{U}{mathx}{m}{n}
\DeclareFontSubstitution{U}{mathx}{m}{n}

\DeclareMathDelimiter{\vvvert}{0}{matha}{"7E}{mathx}{"17}

\newcommand{\vvv}{\vvvert}

\begin{document}
\baselineskip = 14pt

\title[Unique ergodicity for SdNLW on $\T^{2}$]
{On the unique ergodicity for a class of 2 dimensional stochastic wave equations}

\author[J.~Forlano and L.~Tolomeo]
{Justin Forlano and Leonardo Tolomeo}

\address{
Justin Forlano, Maxwell Institute for Mathematical Sciences\\
Department of Mathematics\\
Heriot-Watt University\\
Edinburgh\\
EH14 4AS\\
United Kingdom
and Department of Mathematics\\
University of California\\
Los Angeles\\
CA 90095\\
USA}

\email{forlano@math.ucla.edu}

 \address{
 Leonardo Tolomeo\\
 Mathematical Institute\\
  Hausdorff Center for Mathematics\\
   Universit\"{a}t Bonn\\
   Bonn\\
  Germany
}

\email{tolomeo@math.uni-bonn.de}

\subjclass[2010]{35L15, 37A25, 60H15}

\keywords{stochastic nonlinear wave equation; ergodicity; invariant measure; white noise}

\begin{abstract}
We study the global-in-time dynamics for a stochastic semilinear wave equation with cubic defocusing nonlinearity and additive noise, posed on the $2$-dimensional torus. The noise is taken to be slightly more regular than space-time white noise. In this setting, we show existence and uniqueness of an invariant measure for the Markov semigroup generated by the flow over an appropriately chosen Banach space. This extends a result of the second author \cite{t18erg} to a situation where the invariant measure is not explicitly known.
\end{abstract}

\maketitle
\tableofcontents

\section{Introduction}

Given $s > 0$ and $\g \in \R$, we consider here the following stochastic damped nonlinear wave equation (SDNLW): 
\begin{equation}\pagebreak
\begin{cases}\label{SDNLW}
\dt^{2}u+\dt u+u-\Delta u +u^{3}-3\g u=\sqrt{2}\jb{\nb}^{-s}\xi \\
(u,\dt u)|_{t = 0} = (u_0,u_1)
\end{cases}
\qquad ( t, x) \in \R_{+} \times \T^{2}, 
\end{equation}
where $\xi$ is a space-time white noise, and $\jb{\nb}^{-s}$ denotes the smoothing operator 
\begin{equation*}
\jb{\nb}^{-s} := \big(\tfrac 34- \Delta\big)^{-\frac s2}.
\end{equation*}
We consider \eqref{SDNLW} as a first-order system in the variable $\u := \begin{pmatrix} u \\ u_t \end{pmatrix}$,
\begin{equation}
\begin{cases}
\partial_t \u =  \begin{pmatrix} u \\ u_t \end{pmatrix}  = -\begin{pmatrix} 0 & -1 \\  1-\Delta & 1 \end{pmatrix} \begin{pmatrix} u \\ \dt u \end{pmatrix}  - \vec{0}{u^3-3\g u} +\vec{0}{\sqrt{2} \jb{\nb}^{-s}\xi}, \\ 
\u|_{t = 0} = \u_0 = \begin{pmatrix} u_0 \\ u_1 \end{pmatrix}.
\end{cases}
 \label{NLWvec}
\end{equation}
Our goal is to prove the following theorem.
\begin{theorem}\label{THM:erg}
Given $s > 0$ and $\g \in \R$, there exists a measure $\rho_s$ concentrated on a suitable Banach space $\X$ which is invariant for the Markov process associated to the flow $\Phi_t(\cdot,\xi)$ of \eqref{NLWvec}, in the sense that for every
function $F$ measurable and bounded,
\begin{align*}
\int \E[ F(\Phi_{t}(\mathbf{u}_0, \xi))]d\rho_{s}(\u_0)=\int F(\u_0)d\rho_{s}(\u_0), \,\,\,\, \text{for every} \,\, t>0.
\end{align*}
Moreover, $\rho_{s}$ is the only invariant measure concentrated on $\X$.
\end{theorem} 
By Birkhoff's ergodic theorem, Theorem~\ref{THM:erg} implies that the dynamics are converging to the invariant measure $\rho_s$, in the sense that for every $\u_0 \in \X$, 
\begin{equation} \label{birkhoff}
\lim_{T \to \infty} \frac 1 T \int_0^T F(\Phi_t(\u_0,\xi))dt = \int F(\u)d\rho_{s}(\u) \quad \text{a.s.}
\end{equation}
The definition of the space $\X$ will be provided in Section 2, however, we point out that the space $\X$ is rich enough to contain smooth functions (see Lemma \ref{LEM:Hcompact}). In particular, this implies that the convergence result \eqref{birkhoff} holds for every smooth choice of the initial data $\u_0$, thus allowing to obtain very precise information about the long time behaviour of the flow of \eqref{SDNLW}.

Our interest for the equation \eqref{SDNLW} arises from the recent results pertaining to the PDE construction of the Euclidean $\Phi^4$ quantum field theory. Namely, if one considers the stochastic quantisation equation 
\begin{equation} \label{SQE}
\partial_t u + u - \Delta u = u^3 - C_d \cdot u + \sqrt2\xi, 
\end{equation}
where $C_d$ is a constant depending on the dimension $d$, and $C_d = \infty$ for $d \ge 2$, then,
as time goes to infinity, a generic solution to \eqref{SQE} should converge to the (massive) $\Phi^4$-measure. A result about existence (and uniqueness) of an invariant measure for \eqref{SQE} would provide a new construction of the $\Phi^4$-measure, completing the project started by Parisi and Wu in their influential work \cite{PaWu}. 
This has recently been achieved on the torus in (fractional) dimension $d< 4$ in the papers  \cite{mw1,TW,mw2,cmw}, that show the relevant estimates to ensure the existence of an invariant measure via a Krylov-Bogolyubov argument. Together with the abstract ergodicity result of \cite{HM2}, this also allows to show an analogous version of Theorem~\ref{THM:erg} for the equation \eqref{SQE}. We also point out the papers \cite{gh1,gh2}, which provide a construction of the $\Phi^4_3$-model on the Euclidean space $\R^3$. 

In the above context, the stochastic wave equation 
\begin{equation}\label{CSQE}
\partial_t^2 u + \partial_t u + u - \Delta u = u^3 - C_d \cdot u + \sqrt2 \xi, 
\end{equation}
corresponds to the so-called canonical stochastic quantisation equation, or equivalently, to the Langevin equation for the $\Phi^4$-measure with momentum $v = \partial_t u$. Therefore, it is natural to ask if similar results to the long-time behaviour of solutions to \eqref{SQE} hold for the equation \eqref{CSQE} as well. However, due to the worse analytical properties of the wave propagator $(\partial_t^2 - \Delta)^{-1}$ compared to those of the heat propagator $(\partial_t - \Delta)^{-1}$, there are far fewer results available. 
Indeed, ergodicity for the equation \eqref{CSQE} on the one-dimensional torus has only recently been proved by the second author in \cite{t18erg}. The approach in \cite{t18erg} though heavily uses some structural property of the $\Phi^4_1$-measure and is thus not applicable when the limiting measure is unknown.

Following the approach suggested by \cite{cmw}, where $\gamma = C_d$ is an appropriate constant, the equation \eqref{SDNLW} corresponds to \eqref{CSQE} for fractional dimension $d = 2 - s$.  Equations of the form of \eqref{SDNLW} have already been considered in the literature by Barbu and Da Prato in \cite{BdP}, and Brze\'zniak, Ondrej\'at and Seidler in \cite{bos16}. 
However, when applying the results of \cite{BdP,bos16} to \eqref{SDNLW}, we obtain Theorem~\ref{THM:erg} only for $s > 1$, which would correspond to an ergodicity result for \eqref{CSQE} in fractional dimensions $d < 1$.
Therefore, our result extends the convergence of \eqref{CSQE} to the $\Phi^4$-measure to the full non-singular regime $d < 2$.

\begin{remark}
The equation \eqref{CSQE} for $d=2$ (without the damping term) has first been considered in \cite{GKO}, where Gubinelli, Koch, and Oh showed local well-posedness for smooth initial data. 
In the following paper \cite{GKOT}, Gubinelli, Koch, Oh and the second author showed global well-posedness for the same equation, and proved that the $\Phi^4_2$-measure is invariant for the semigroup generated by the flow of \eqref{CSQE}. 
However, the estimate for the growth of the solution grows with a double exponential in time. 
Compare this with the estimate \eqref{vmoments}, which is bounded as time goes to infinity. 
The lack of such a good estimate makes the arguments of this paper break down, with the main culprit being the $L^2$-estimate \eqref{hL2}.
These issues have recently been solved in the work in preparation \cite{Tolomeo3}, where building upon both the techniques of \cite{t18erg} and the techniques developed in this paper, the second author proves ergodicity of the measure $P(\Phi)_2$ for the equation
\begin{equation*}
\partial_t^2 u + \partial_t u + u - \Delta u + p(u) = \sqrt2\xi,
\end{equation*}
where $p(u)$ is any (appropriately renormalised) polynomial with odd degree and positive leading coefficient.
\end{remark}

\subsection{Mild formulation and construction of the flow}
In order to construct the flow of the equation \eqref{SDNLW}, we need to define what we mean by a solution of \eqref{SDNLW}. If one considers the linear equation,
\begin{equation}\label{SDNLWlin}
\begin{cases}
\dt^{2}u+\dt u+u-\Delta u = f(t) \\
(u,\dt u)|_{t = 0} = \u_0^T = (u_0,u_1)
\end{cases}
\qquad ( t, x) \in \R_{+} \times \T^{2}, 
\end{equation}
by variation of constants, it is easy to show that the solution to \eqref{SDNLWlin} satisfies 
\begin{equation*}
\u(t) = S(t) \u_0 + \int_0^t S(t-t') \vec{0}{f(t')} d t',
\end{equation*}
where $S(t)$ is the linear propagator for the (damped) wave equation, and it is given by the formula
\begin{equation} \label{linsol}
e^{-\frac t2}
\begin{pmatrix}\cos\big(t\jb{\nb}\big) + \frac12\frac{\sin(t\jb{\nb})}{\jb{\nb}}
&\frac{\sin(t\jb{\nb})}{\jb{\nb}} \\
-\big(\jb{\nb}+\frac1{4\jb{\nb}}\big)\sin(t\jb{\nb}) & 
\cos(t\jb{\nb}) -\frac12\frac{\sin(t\jb{\nb})}{\jb{\nb}}
\end{pmatrix}.
\end{equation}

Motivated by this, we expect the solution to \eqref{SDNLW} to be given by 
\begin{align}
\u= S(t)\u_0 +\int_{0}^{t}S(t-t') \begin{pmatrix} 0 \\ {\text{\small{$\sqrt{2}$}}}\jb{\nb}^{-s}\xi(t') \end{pmatrix} dt' -\int_{0}^{t}S(t-t')\begin{pmatrix} 0 \\ \pi_{1}( \u^{3}(t')-3\g \u(t')) \end{pmatrix}dt',
\label{duhamel}
\end{align}
where $\pi_1$ denotes the projection to the first coordinate.
We see from \eqref{duhamel} that a natural object to study is 
\begin{align}
\stick_t(\xi) :=  \int_0^t S(t-t')\vec{0}{\sqrt{2}\jb{\nb}^{-s}\xi(t')} dt',
\label{stick}
\end{align}
which is a (random) space-time distribution known as the \emph{stochastic convolution}. It is the solution to the linear damped wave equation \eqref{SDNLWlin} with $f = \xi$ and $\u_0 = \0$. Since $\xi(t')$ is only a space-time distribution, the definition \eqref{stick} does not immediately make sense. We will define the object $\stick_t(\xi)$ precisely and prove its relevant properties in Section 2.
We will sometimes write $\stick_{t}$ in place of $\stick_{t}(\xi)$, where confusion will not arise. We point out that, for each fixed $t\geq 0$, the stochastic convolution $\stick_t$ belongs to the H\"older space $C^\alpha \times C^{\alpha -1}$ only for $\alpha < s$. Therefore, the solution $\u$ to \eqref{SDNLW} will also be a low-regularity function, with $\u(t) \in H^{\alpha} \times H^{\alpha-1}$ only for $\alpha < s$. 
We note that due to the unboundedness of the linear propagator $S(t)$, it is not possible to develop a solution theory in H\"older spaces $C^\alpha \times C^{\alpha-1}$.

Motivated by \eqref{duhamel} and these considerations about the stochastic convolution~\eqref{stick}, we define a solution of $\eqref{SDNLW}$ in terms of the first-order expansion~\cite{mckean, Bourgain2, dpd}:
\begin{align} \label{soldef}
\u(t) = \Phi_t(\u_0,\xi) =S(t)\u_0 +\stick_{t}(\xi)+\v(t),
\end{align}
where $\v=\begin{pmatrix} v \\ \dt v \end{pmatrix}$ solves the equation
\begin{align}
\v(t) = - \int_0^t S(t-t') \vec{0}{\NN_{\g}[\pi_1( S(t')\u_0+\stick_{t'}(\xi)+\v(t'))]  } dt',
\label{veqn}
\end{align}
with 
\begin{align*}
\NN_{\g}[v]:=v^{3}-3\g v.
\end{align*}
After appropriately defining the space $\X$ in Section 2, this decomposition allows us to prove local well-posedness of \eqref{SDNLW} on the space $\X$, while providing at the same time the extra information that $\v$ belongs to the Sobolev space $H^1 \times L^2$. See Proposition \ref{PROP:LWP} for a precise statement. In Section 3.2, we show the following global well-posedness result. 
\begin{theorem} \label{THM:GWP}
Let $\u_0 \in \X$. Then for every $T > 0$, there exists a unique function \linebreak $\u:[0,T] \to \X$ such that $\u$ solves \eqref{SDNLW} in the sense of \eqref{soldef}, and moreover
\begin{equation} \label{gwpintro}
\sup_{0 \le t \le T} \| \v(t) \|_{H^1 \times L^2} < +\infty \quad \text{a.s. }
\end{equation}
\end{theorem}
See Proposition \ref{PROP:GWP} for a more precise statement. Following the techniques introduced by \cite{BTglobal}, the proof of this theorem will rely on an energy estimate based on the functional 
\begin{align*}
\begin{split}
E(\v):= \frac{1}{2}\int (\dt v)^{2} +\frac{1}{2}\int v^2 +\frac{1}{2}\int |\nb v|^{2} +\frac{1}{4}\int v^4 +\frac{1}{8}\int (v+\dt v)^{2},
\end{split}
\label{energy}
\end{align*}
together with an integration by part tricks introduced in \cite{BOP3}. Similar techniques are used in the paper by the second author \cite{t18erg}. However, we point out that due to the lack of the Sobolev embedding $H^1 \embeds L^\infty$ (or similarly, of the embedding $\X \embeds L^\infty$) the analysis will be more complicated than the one in \cite{t18erg}.

\subsection{Existence and uniqueness of the invariant measure}
Due to the actual definition of the solutions to \eqref{SDNLW}, i.e.\ the decomposition \eqref{soldef}, it is a priori unclear if the operators 
\begin{equation*}
\Pt{t} F (\u_0) := \E[F(\Phi_t(\u_0,\xi))]
\end{equation*}
actually define a Markov semigroup on bounded functions $F:X^{\al}\to \R$. The proof of this is carried out in Section 4 (see Proposition \ref{PROP:markov}). 
This allows us to use standard arguments to show the existence of an invariant measure $\rho_s$. Indeed, in Proposition \ref{PROP:KryBo}, we implement the estimate coming from \eqref{gwpintro}, together with some considerations expressed in Section 2, into a Krylov-Bogolyubov argument. 

However, similarly to what happens in \cite{t18erg}, the flow does not satisfy the strong Feller property (see Proposition \ref{PROP:notsf}). This is due to the fact that the operator $S(t)$ is invertible on the Sobolev spaces $H^{\s} \times H^{\s -1}$, which is a big difference with the case of parabolic stochastic equations. Therefore, the techniques of \cite{TW, HM2} are not available to us.
In order to deal with this problem, we adapt the approach of the work \cite{HM1}. In this influential paper, Hairer and Mattingly introduced the concept of ``asymptotic strong Feller", exactly to deal with Markov semigroups that do not have the strong Feller property.
For two probabilities $\mu_1, \mu_2$, we consider the family of distances 
\begin{equation} \label{dndef}
d_n(\mu_1,\mu_2) := \inf_{\pi \in \Pi(\mu_1,\mu_2)} \int_{\X \times \X} 1 \wedge n\| x - y\|_{\X} d \pi(x,y), 
\end{equation}
where $\Pi(\mu_1,\mu_2)$ is the family of probability measures on $\X \times \X$ with marginals given by $\mu_1$ and $\mu_2$, respectively.
In order to build a plan $\pi$ to test the infimum in \eqref{dn}, we construct a so-called asymptotic coupling. More precisely, we adapt the strategy delineated in \cite{HM2} in order to prove the strong Feller property, 
and show that for every $\u_1^0, \u_2^0 \in \X$, there exists a shift $h \in L^2_{t,x}(\R_{+}\times \T^2)$, adapted with respect to the filtration induced  by $\xi$, such that 
\begin{gather*}
\|\Phi_{t}(\u_1^0, \xi ) - \Phi_{t}(\u_2^0, \xi + h) \|_{\X} \le e^{-ct} \|\u_1^0 - \u_2^0\|_{\X},\\
\E \big[ \| h \|_{L^2_{t,x}(\R_{+} \times \T^2)}^2 \big]\le C(\u_1^0, \u_2^0) \|\u_1^0 - \u_2^0\|_{\X}.
\end{gather*}
This is carried out in Proposition \ref{PROP:shift} and Lemma \ref{LEM:h}. In Section \ref{SEC:proofThm1}, we rewrite \eqref{dndef} using Kantorovich duality, and we use the existence of such an $h$ to show unique ergodicity for the measure $\rho_s$. Moreover, adapting the arguments of \cite{HM3}, we can show the following convergence result.
\begin{theorem} \label{THM:conv}
For every $\u_0 \in \X$, and for every $n \in \mathbb{N}$, we have that 
\begin{equation} \label{EQN:WassConv}
\lim_{t \to \infty} d_n(\Pt{t}^\ast \delta_{\u_0}, \rho_s) = 0.
\end{equation}
\end{theorem}
The limit \eqref{EQN:WassConv} can be rewritten in the following suggestive way: for every $\u_0 \in \X$, there exists a stochastic process $X_{\u_0}(t)$ such that $\L(X_{\u_0}(t)) = \rho_s$, and 
\begin{equation*}
\lim_{t \to \infty} \| \Phi_t(\u_0, \xi) - X_{\u_0}(t) \|_{\X} = 0 \hspace{7pt} \text{a.s.}
\end{equation*}
We point out that the convergence in \eqref{EQN:WassConv} is not quantitative. While it would be a priori possible to obtain quantitative estimates for it, because of the exponential term in the long time estimate \eqref{vmoments}, 
we do not expect such estimates to be very strong, with the resulting convergence rate being \emph{at best} sub-polynomial (see Remark \ref{convSpeed}). Compare this with the results of papers such as \cite{HM3,TW}, that show \emph{exponential} convergence to equilibrium. Because of the damping term in the equation \eqref{SDNLW}, we would expect the convergence rate to be exponential in this case as well, however, because of the low regularity of the problem (when $s$ approaches 0), we consider this to be out of reach with the current techniques. 
Therefore, we decided to limit the technical complexity of the paper, and not to pursue quantitative estimates for \eqref{EQN:WassConv}.

\begin{remark}
Despite heavily relying on the techniques of \cite{HM1,HM3}, we do not explicitly prove the asymptotic strong Feller property for the semigroup $\Pt{t}$, nor irreducibility of the flow. 
The asymptotic strong Feller property is an easy consequence of Proposition \ref{PROP:flowclose}, while irreducibility can be proven with similar arguments to the ones used in this paper. 
However, as it turns out, the estimate \eqref{diffflow} in Proposition \ref{PROP:flowclose} is enough for the proofs of Theorems \ref{THM:erg} and \ref{THM:conv}. Therefore, we decided to follow a more direct approach, and not to refer explicitly to these concepts.
\end{remark}

\begin{remark}\rm \label{RMK:2DSDNLW}
In~\cite{t18erg}, the second author studied a class of stochastic damped wave equations with cubic defocusing nonlinearity and space-time white noise forcing. In particular, the following equations were considered on $\T^2$:
\begin{align}
\partial_t^2 u + \partial_t u + u +(-\Delta)^{1+\frac{s}{2}}\, u +u^3 =\sqrt{2}\, \xi, \label{2DSDNLW}
\end{align} 
where $s>0$, the initial data is distributed according to the invariant Gibbs measure 
\begin{align}
d\rho(u,u_{t})``="Z^{-1} \exp \bigg(  -\frac{1}{2}\int u^2+u_{t}^{2}+| (-\Delta)^{\frac{1}{2}+\frac{s}{4}}u|^2 dx-\frac{1}{4}\int u^4 dx\bigg) du du_{t}, \label{Gibbs}
\end{align}
and we have restricted to two spatial dimensions for easier comparison to \eqref{SDNLW}. 
The main result in \cite{t18erg} is the uniqueness of the Gibbs measure \eqref{Gibbs} under the flow of \eqref{2DSDNLW}. We note that the approach in this paper provides an alternative proof of this result; see Remark~\ref{RMK:t18erg} for more details. However, this argument crucially relies on the good long time bounds for the flow given in Proposition \ref{PROP:GWP}, whilst the approach in~\cite{t18erg} proves ergodicity of the invariant measure \eqref{Gibbs} without the need of such quantitative bounds. 
\end{remark}

\noi
{\bf Notations.}
For a function $f: \T^2 \sim [0,1]^2 \to \C$, its Fourier series is given by $\widehat f: \Z^2 \to \C$, where
\begin{equation*}
\widehat f (n) := \int_{\T^2} f(x) e^{-2\pi i n \cdot x} d x.
\end{equation*}
This leads to the inversion formula 
\begin{equation*}
f(x) := \sum_{n\in\Z^2} \widehat f(n) e^{2\pi i n \cdot x}.
\end{equation*}
Moreover, for a function $F: \C^2 \to \C$, the expression $F(\nabla)u$ will denote the distribution with Fourier series given by
\begin{equation*}
\widehat{F(\nabla) u}(n) := F(2 \pi i n) \widehat u (n).
\end{equation*}
For $x \in \R^n$, we denote
\begin{equation*} 
\jb{x} := \big(\tfrac 34 + |x|^2\big)^\frac 12.
\end{equation*}
In order to make many of the computations of this paper rigorous, it is necessary to consider a suitable finite dimensional truncation of \eqref{NLWvec}. Given $N\in \N\cup \{0\}$, we define $P_{\leq N}$ as the sharp projection onto frequencies in the square $[-N,N]^{2}$, that is, 
\begin{align*}
P_{\leq N}\u(x)= \frac{1}{(2\pi)^{2}} \sum_{\substack{n\in \Z^{2}\\ \max_{j}|n_j|\leq N}} \ft{\u}(n)e^{in\cdot x}.
\end{align*}
We also define $P_{>N}:=\text{Id}-P_{\leq N}$ and for $N=-1$, extend these projections to $P_{\leq -1}=0$ and $P_{>-1}=\text{Id}$.
With a slight abuse of notation, we define the truncated system of \eqref{NLWvec} by 
\begin{equation}
\begin{cases}
\partial_t \vec{u}{\dt u}  &= -\begin{pmatrix} 0 & -1 \\  1-\Delta & 1 \end{pmatrix} \begin{pmatrix} u \\ \dt u \end{pmatrix}  - P_{\leq N} \vec{0}{(P_{\leq N}u)^3-3\g P_{\leq N}u} = \vec{0}{\sqrt{2} \jb{\nb}^{-s}\xi} \\
 \begin{pmatrix} u \\ \dt u \end{pmatrix}(0) & = \begin{pmatrix} u_0 \\ u_1 \end{pmatrix}.
\end{cases}
 \label{NLWvectrunc}
\end{equation}
Writing solutions $\u_{N}$ to \eqref{NLWvectrunc} as 
\begin{align*}
\u_N =S(t)\u_0 +\stick_{t}(\xi)+\v_N,
\end{align*}
we see that $\v_N$ solves 
\begin{align}
\v_{N}(t) = - \int_0^t S(t-t')P_{\leq N} \vec{0}{\NN_{\g}[\pi_1 P_{\leq N}(S(t')\u_0+\stick_{t'}(\xi)+\v_{N}(t'))]  } dt'.
\label{veqnN}
\end{align}

\section{Preliminaries}
We recall the following version of the fractional Leibniz inequality, which will be useful throughout the paper.
\begin{lemma}[Lemma 3.4, \cite{GKO}]\label{LEM:fraclieb}
Let $0\leq \s\leq 1$. Suppose that $1<p_j, q_j ,r<\infty$, $\frac{1}{p_{j}}+\frac{1}{q_j}=\frac{1}{r}$, $j=1,2$. Then, we have 
\begin{align}
\| \jb{\nb}^{\s}(fg)\|_{L^{r}(\T^2)}\les \| \jb{\nb}^{\s}f\|_{L^{p_1}(\T^2)}\|g\|_{L^{q_1}(\T^2)}+\|f\|_{L^{p_2}(\T^2)}\|\jb{\nb}^{\s}g\|_{L^{q_2}(\T^2)}. \label{fraclieb}
\end{align}
\end{lemma}

\subsection{Function spaces}
We define the  Sobolev spaces $W^{\al,p}$ and $\W^{\al,p}:=W^{\al,p}\times W^{\al-1,p}$ via the norms
\begin{align*}
\| u\|_{W^{\al,p}}:= \| \jb{\nabla}^\al u\|_{L^p} \quad \text{and} \quad \| \u \|_{\W^{\al, p}}^p := \| \jb{\nabla}^\al u\|_{L^p}^p +  \| \jb{\nabla}^{\al-1}u_t\|_{L^p}^p,
\end{align*}
with the usual modification when $p=\infty$. Moreover, when $p = 2$, we write $\H^\al := \W^{\al,2}$, and when $p = \infty$,  $\mathcal{C}^{\al}:=C^{\al}\times C^{\al-1}:=W^{\al,\infty}\times W^{\al-1,\infty}$. With this definition and recalling the formula \eqref{linsol}, it is easy to check that 
\begin{align}
\| S(t)\u \|_{\H^{\al}} \les e^{-\frac{t}{2}}\| \u \|_{\H^{\al}}. \label{SonH}
\end{align}

In order to prove the result of Theorem \ref{THM:erg}, we need $\X$ to be a suitable space, such that the invariant measures for the flow of \eqref{SDNLW} concentrate on $\X$.
 Motivated by the choice of spaces introduced in \cite{t18erg}, we define, for $0<\al < 1$, the space \begin{align*}
\overline{X}^{\al}:=\{ \u \in \H^\al \, : \, S(t)\u \in C([0,+\infty); \W^{\al, \frac{2}{\al}}), \, \|S(t)\u\|_{\Wal}\les e^{-\frac{t}{8}} \},
\end{align*}
with the norm
  \begin{align}
  \| \u \|_{\X}:=\sup_{t\ge0} \, e^{\frac{t}{8}}\|S(t)\u\|_{\W^{\al,\frac{2}{\al}}}.
  \label{Xalpha}
\end{align}
We note here that later on we will further restrict the upper bound on $\al$. 
The space $\overline{X}^{\al}$ is a Banach space (see \cite[Lemma 1.2]{t18erg}), however, it is unclear if $\overline{X}^{\al}$ is separable or not. In order to avoid this problem, we simply define $X^{\al}$ to be the closure of trigonometric polynomials in $\overline{X}^{\al}$ under the norm $\|\cdot \|_{\X}$.
Note that for any $t\geq 0$, 
\begin{align}
\| S(t)\|_{X^{\al}\mapsto \X}\leq e^{-\frac{t}{8}}. \label{S(t)onX}
\end{align}
To be consistent with the Sobolev embedding $H^1 \embeds L^6$, which will be useful for analysing \eqref{veqn}, we will further restrict to $0<\al<\frac{1}{3}$. Note, however, that the following result is true for any $0<\al<1$.

\begin{lemma}\label{LEM:Hcompact}
For any $0<\al<\frac{1}{3}$, $\H^{1}\embeds \X$ and the embedding $\H^{1+\al}\embeds \X$ is compact.
\end{lemma}
\begin{proof}
For $\u \in \H^1$ and any $t\geq 0$, Sobolev embedding and \eqref{SonH} imply  
\begin{align*}
\| S(t)\u\|_{\Wal} \les \|S(t)\u\|_{\H^{1}} \les e^{-\frac{t}{2}}\|\u \|_{\H^1}. 
\end{align*} 
Then, for any $t_0 \geq 0$, since $S(t)\u \in C([0,\infty);\H^{1})$, we have 
\begin{align*}
\lim_{t\to t_0} \| S(t)\u-S(t_0)\u\|_{\Wal} \les \lim_{t\to t_0}\| S(t)\u-S(t_0)\u\|_{\H^1}=0.
\end{align*}
Moreover, since the embedding $\H^{1+\al} \embeds \H^1$ is compact, we have that  $\H^{1+\al} \embeds \H^1 \embeds \X$ is compact as well, by composition.
\end{proof}

\begin{remark} \rm \label{Xnecessity}
One might wonder if the definition of the space $\X$ is truly necessary, and if it would suffice to work with a standard Sobolev space instead. However, as it will be clear from the arguments in Section 3, any candidate space $X$ must satisfy the inequalities 
$$ \| S(t) \u\|_{L^6} \les \| \u \|_{X} \quad \text{and} \quad \| S(t) \u\|_{X} \les \| \u \|_{X}. $$
By the properties of the linear propagator for the wave equation, 
if we were to use standard Sobolev spaces,
 the second inequality would force $X = \H^{\alpha'}$ for some $\alpha' \in \R$, 
 and the first inequality then forces $\alpha' \ge \frac 23$. 
 However, as already mentioned in the introduction, 
 we expect the solution to \eqref{SDNLW} to be a relatively rough function, namely, 
 $\u \not\in \H^{\alpha}$ for every $\alpha \ge s$.
 Therefore, our space $X$ must be big enough to contain truly rough functions. 
 By the same argument as in \cite[Lemma 1.4]{t18erg}, we have for every $\al_1>\al>0$, 
 that there exists $\u_0\in \X$ such that $\u_0\notin \H^{\al_1}$, 
 so at least from this point of view, the spaces $\X$ appear as a natural candidate for our well-posedness theory.
\end{remark}

Similar to the spaces $\cj{X}^{\al}$ and $\X$, for $0<\al<1$, we define 
\begin{align*}
\overline{Z}^{\al}:=\{ \u \in \H^{\al} \, :\, S(t)\u \in C([0,+\infty); \CC^{\al}), \,\, \| S(t)\u\|_{\CC^{\al}}\les e^{-\frac{t}{8}}\, \},
\end{align*}
and define $Z^{\al}$ to be the closure of trigonometric polynomials in $\cj{Z}^{\al}$ under the norm 
\begin{align*}
\|\u\|_{Z^{\al}}=\sup_{t\geq 0} e^{\frac{t}{8}}\|S(t)\u\|_{\CC^{\al}}.
\end{align*}
Now, $Z^{\al}$ is separable and is a Banach space (see \cite[Lemma 1.2]{t18erg}) and clearly $Z^{\al}\embeds X^{\al}$. Notice that, by definition of the operator $S(t)$ in \eqref{linsol} and of the Sobolev spaces $\W^{\al,p}$, we have that 
\begin{equation*}
\Big\| \partial_t S(t) \u \Big\|_{\W^{\be,q}} \les \|S(t) \u\|_{\W^{\be+1,q}}.
\end{equation*}
Therefore, for $0\leq t_1 \le t_2$, 
\begin{equation*}
\Big\| \jb{\nabla}^{-1} \frac{S(t_1) - S(t_2)}{|t_1 - t_2|} \u \Big\|_{\W^{\be,q}} \les \sup_{t_1 \le t \le t_2} \|S(t) \u\|_{\W^{\be,q}},
\end{equation*}
and by interpolation, for every $0 \le \ta \le 1$ and $\ep > 0$, we have that  
\begin{gather}
\Big\| \frac{S(t_1) - S(t_2)}{|t_1 - t_2|^\ta} \u \Big\|_{\W^{\be,q}} \les  \sup_{t_1 \le t \le t_2}  \|S(t) \u\|_{\W^{\be+\ta,q}}, \label{derivTrade}\\
\Big\| \frac{S(t_1) - S(t_2)}{|t_1 - t_2|^\ta} \u \Big\|_{\CC^{\be}} \les \sup_{t_1 \le t \le t_2}  \|S(t) \u\|_{\CC^{\be+\ta+\ep}} \label{derivTrade1}.
\end{gather}
We have the following compactness property. 
\begin{lemma}\label{LEM:Zcompact}
Let $0<\al<\s <1$. Then, the embedding $Z^{\s}\embeds \X$ is compact.
\end{lemma}
\begin{proof}
Let $\{ \u_n\}_{n\in \N}$ be a sequence in $Z^{\s}$ such that $\sup_{n\in \N}\|\u_n\|_{Z^{\s}}\leq C.$
Then, we have 
\begin{align}
\sup_{n\in \N} \|S(t)\u_n\|_{\CC^{\s}}\leq Ce^{-\frac{t}{8}} \label{ZX1}
\end{align}
for every $t\geq 0$. By compactness of the embedding $\CC^{\s}\embeds \Wal$, for every $t\geq 0$, there exists $\v(t) \in \Wal$ such that, up to subsequences, $S(t)\u_n$ converges to $\v(t)$ in $\Wal$. Now, by \eqref{SonH} and \eqref{ZX1}, we have 
\begin{align}
\|\u_n\|_{\H^{\s}}& \les e^{\frac{1}{2}}\|S(1)\u_n\|_{\H^{\s}}\les e^{\frac{1}{2}}C, \label{ZX2}
\end{align}
and hence, up to further subsequences, there exists $\u\in \H^{\s}$ such that $\u_n \wto \u$ in $\H^{\s}$ and thus, also $S(t)\u_n \wto S(t)\u$ in $\H^{s}$, for every $s\leq \s$ and $t\geq 0$. Note that \eqref{ZX2} implies
\begin{align}
\|\u\|_{\H^{\s}}\les e^{\frac{1}{2}}C. \label{ZX3}
\end{align}
Meanwhile, from \eqref{ZX1}, there exists $\w(t)\in \CC^{\s}$ such that, up to further subsequences, $S(t)\u_n$ converges in the weak-$\ast$ topology to $\w(t)$ in $\CC^{\s}$ and thus also converges weakly in $\H^{s}$. Identifying limits implies $\v(t)=\w(t)=S(t)\u$
and thus $S(t)\u\in \CC^{\s}$ and $S(t)\u_n$ converges to $S(t)\u$ in $\Wal$. By weak-$\ast$ lower semicontinuity of $\CC^{\s}$, we also have 
\begin{align}
\|S(t)\u\|_{\CC^{\s}}\leq Ce^{-\frac{t}{8}} \label{ZX4}
\end{align}
for any $t\geq 0$. It remains to show $\u\in \X$ and $\u_n \to \u$ in $\X$. By interpolation, \eqref{derivTrade1}, \eqref{ZX1}, \eqref{ZX2}, \eqref{ZX3} and \eqref{ZX4}, for any $t_2\geq t_1>0$,
\begin{align*}
\|S(t_2)\u-S(t_1)\u\|_{\Wal}& \les \|S(t_2)\u-S(t_1)\u\|_{\CC^{\al}}^{1-\al}\|S(t_2)\u-S(t_1)\u\|_{\H^{\al}}^{\al} \\
& \les C^{1-\al}|t_2-t_1|^{\ep \al} \|\u\|_{\H^{\al+\ep}}^{\al} \\
&\les C|t_2-t_1|^{\ep \al},
\end{align*}
where $0<\ep <\s -\al$. This shows that $S(t)\u \in C([0,+\infty); \Wal)$. Applying this same reasoning to $\{ S(t)\u_n\}_{n\in \N}$, we get 
\begin{align*}
\sup_{n\in \N}\|S(t_2)\u_n -S(t_1)\u_n\|_{\Wal}\les C|t_2-t_1|^{\ep \al},
\end{align*}
which implies that $S(t)\u_n$ converges to $S(t)\u$ in $\Wal$ uniformly on compact sets. Now for any fixed $T>0$, \eqref{ZX1}, \eqref{ZX2}, \eqref{ZX3} and \eqref{ZX4} imply
\begin{align*}
e^{\frac{t}{8}}\|S(t)\u_n & -S(t)\u\|_{\Wal} \\
& \les e^{\frac{T}{8}}\sup_{\tau \in [0,T]}\|S(t)\u_n -S(t)\u\|_{\Wal} +\sup_{t\in (T,+\infty)}e^{\frac{t}{8}}\|S(t)\u_n -S(t)\u\|_{\Wal} \\
& \les e^{\frac{T}{8}}\sup_{\tau \in [0,T]}\|S(t)\u_n -S(t)\u\|_{\Wal}+Ce^{-\frac{3\al T}{8}}.
\end{align*}
Thus, taking $T\gg 1$ and $n\gg 1$ sufficiently large, depending on $T$, the above shows that $\u_n \to \u$ in $\X$, which completes the proof.
\end{proof}

\begin{lemma}\label{compactOrbits}
For any $\u_0 \in \X$, let 
\begin{equation}\label{orbit}
O_{\u_0} := \{ S(t)\u_0 : t \ge 0 \}\cup\{\0\}\subset \X.
\end{equation}
Then $O_{\u_0}$ is compact in $\X$.
\end{lemma}
\begin{proof}
We first show that the map $t \mapsto S(t) \u_0$ is continuous in $t$. If $\u_0$ is a trigonometric polynomial, this follows easily from the embedding $\H^1 \embeds \X$. If $\u_0$ is not a trigonometric polynomial, we take a sequence of trigonometric polynomials $\u_n$ which is converging to $\u_0$ in $\X$. Notice that trigonometric polynomials are dense in $\X$ by definition. For $t_1, t_2 \ge 0$, using \eqref{S(t)onX}, we have
\begin{align*}
&\phantom{\le\ }\| S(t_1)\u_0 - S(t_2) \u_0 \|_{\X} \\
&\le \| S(t_1)\u_n - S(t_2) \u_n \|_{\X} + \|S(t_1)(\u_0 - \u_n) \|_{\X} + \|S(t_2)(\u_0 - \u_n) \|_{\X}\\
&\le  \| S(t_1)\u_n - S(t_2) \u_n \|_{\X} + 2 \|\u_0 - \u_n\|_{\X}.
\end{align*}
Therefore,
\begin{align*}
\limsup_{t_2 \to t_1} \| S(t_1)\u_0 - S(t_2) \u_0 \|_{\X} \le 2 \|\u_0 - \u_n\|_{\X} \to 0
\end{align*}
as $n \to \infty$.

We now move to proving compactness of $O_{\u_0}$. We notice that $O_{\u_0}$ is the image of the interval $[0,1]$ under the map 
\begin{equation*}
g(\tau) =
\begin{cases}
S\big(\tan(\frac{\pi \tau}{2})\big) \u_0 & \text{ if } 0\le \tau < 1,\\
\0 & \text{ if } \tau=1.
\end{cases}
\end{equation*}
Therefore, it is enough to show that $g$ is continuous. Continuity for $0 \le \tau < 1$ follows from the continuity of $ t \mapsto S(t) \u_0$. Continuity in $\tau = 1$ follows from the estimate
\begin{equation*}
\| S(t) \u_0\|_{\X} \le e^{-\frac t8} \| \u_0\|_{\X} \to 0
\end{equation*}
as $t \to \infty$.
\end{proof}

\subsection{The stochastic convolution}
In this section, we establish regularity properties of the stochastic convolution $\stick_{t}(\xi)$. To this end, we will use the following result, a proof of which can be found in \cite[Proposition A.3.1 and A.3.2]{t19thesis}.

\begin{proposition}\label{PROP:REG}
Let $M=\T^{k}\times \R^{d-k}$ and $X$ be a distribution-valued random variable on $\R_{t}\times M_{x}$.
Let $I\subset \R$ be a closed interval and let $B\subset M$ be a ball of radius $1$. Suppose the following statement holds:
\begin{enumerate}[\normalfont (i)]
\item there exist constants $A_{B}, C_{p}>0$, independent of $t\in I$, and $s\in \R$ such that for every test function $\phi:M\mapsto \R$ supported in $B$, we have 
\begin{align}
\E[| \jb{X(t),\phi}|^{2}]& \leq A_{B}^2 \|\phi\|_{H^{-s}}^{2}, \label{i1} \\
\E[|\jb{X(t),\phi}|^{p}]& \leq C_{p}^{p}\E[|\jb{X(t),\phi}|^{2}]^{\frac{p}{2}}, \label{i2}
\end{align}
for every $p\geq 2$.
\end{enumerate}
Then, for every compact $K\subseteq M$, for every $\ep>0$ and for each fixed $t\in I$, $X(t)\in C^{s-\frac{d}{2}-\ep}(K)$ almost surely. Suppose, in addition, the following holds:
\begin{enumerate}[\normalfont (ii)]
\item there exist $0<\ta<1$ and constants $\wt{A}_{B}, \wt{C}_{p}>0$, independent of $t\in I$, and $s\in \R$ such that for every test function $\phi:M\mapsto \R$ supported in $B$ and for every $t_1, t_2\in I$ such that $ t_1\leq t_2$, we have 
\begin{align}
\E[| \jb{X(0),\phi}|^{2}]& \leq \wt{A}_{B}^2 \|\phi\|_{H^{-s}}^{2}, \label{ii1}\\
\E[|\jb{X(0),\phi}|^{p}]& \leq \wt{C}_{p}^{p}\E[|\jb{X(0),\phi}|^{2}]^{\frac{p}{2}}, \label{ii2}\\
\E[|\jb{X(t_2)-X(t_1),\phi}|^{2}]&\leq \wt{A}_{B}^{2}|t_2-t_1|^{2\ta}\| \phi \|_{H^{-s}}^{2}, \label{ii3}\\
\E[|\jb{X(t_2)-X(t_1),\phi}|^{p}]&\leq \wt{C}_{p}^{p}\E[ |\jb{X(t_2)-X(t_1),\phi}|^{2}]^{\frac{p}{2}}, \label{ii4}
\end{align}
for every $p\geq 2$.
\end{enumerate}
Then, for every compact $K\subseteq M$ and $\ep>0$, $X\in C(I; C^{s-\frac{d}{2}-\ep}(K))$ almost surely.
Moreover, if $X$ is supported in a single ball, then we have 
\begin{align}
\E\bigg[ \|X\|_{C(I;C^{s-\frac{d}{2}-\ep}(K))}^{p}\bigg]^{\frac{1}{p}}\leq C\tilde{A}_{B} \label{momentbd}
\end{align}
for every $p> \frac{2d}{\ep}$, where $C$ depends only on $p$ and $|I|$, the length of the interval $I$.
\end{proposition}

The space-time white noise $\xi$ is the random distribution such that $\{ \jb{\xi,\phi} \}_{\phi\in L^2(\R\times \T^2)}$ is a family of centred Gaussian random variables with variance $\|\phi\|_{L^2(\R\times \T^2)}^2$ on a probability space $(\O, \F, \prob)$ and with the property that 
\begin{align}
\E\big[ \jb{\xi, \phi} \jb{\xi, \psi}  \big]=\jb{\phi, \psi}_{L^2(\R\times \T^2)}  \label{whiteprop}
\end{align}
for any $\phi, \psi\in L^2(\R\times \T^2)$. 
For $t\geq 0$, we set $\tilde{\F}_{t}=\s\big( \big\{ \jb{\xi,\phi} \, : \, \phi\vert_{(t,+\infty)\times \T^2}\equiv 0, \, \phi\in L^2(\R\times \T^2)\big\}\big)$ and we denote by $\{\F_{t}\}_{t\geq 0}$ the usual augmentation of the filtration $\{ \tilde{\F}_{t}\}_{t\geq 0}$ (see~\cite[p. 45]{RevuzYor}).

We now verify the regularity properties of the stochastic convolution $\stick_{t}(\xi)$.
In order to contain $\stick_{t}(\xi)$ within the space $\X$, we will further restrict to $0<\al <s \wedge \frac{1}{3}$. 

\begin{proposition}\label{PROP:REGSTICK}
For every $\al<s$, 
\begin{align*}
\stick_{t}(\xi)\in C([0,+\infty); \CC^{\al}(\T^2)) 
\end{align*}
almost surely.
 In particular, $\stick_{t}(\xi)\in C([0,+\infty); \W^{\al,p}(\T^2))$ for any $1\leq p\leq \infty$, almost surely.
\end{proposition}

\begin{proof}
We verify conditions (i) and (ii) in Proposition~\ref{PROP:REG} for $\pi_1 \stick_{t}(\xi)$ and $\pi_{2}\stick_{t}(\xi)$.
For any test function $\phi$ and $0\leq t_1\leq t_2\leq T$, $\jb{\pi_{j}\stick_{t_1}(\xi),\phi}$ and $\jb{\pi_{j}(\stick_{t_2}(\xi)-\stick_{t_1}(\xi)),\phi}$, $j=1,2$, are Gaussian and hence satisfy \eqref{i2} and \eqref{ii4} with $C_{p}=\wt{C}_{p}=(p-1)^{\frac{1}{2}}$. As $\stick_{0}(\xi)=0$, \eqref{ii1} and  \eqref{ii2} hold trivially. It remains to verify \eqref{i1} and \eqref{ii3}. We begin with the former. 

Fix $t>0$ and let $\phi_1$ and $\phi_2$ be test functions with support in a ball $B$ of radius $1$ and let $\phi=\vec{\phi_1}{\phi_2}$. Then, from \eqref{stick}, we have
\begin{align*}
\jb{\stick_{t}(\xi),\phi}_{L^{2}_{x}}& ={\sqrt{2}} \jb{ \xi(t'), \ind_{\{ 0\leq t'\leq t\}}\jb{\nb}^{-s}\pi_2 S(t-t')^{\ast}\phi}_{L^{2}_{x,t'}}.
\end{align*}
Thus, by \eqref{whiteprop},
\begin{align*}
\begin{aligned}
\E[ | \jb{\stick_{t}(\xi), \phi}_{L^2_{x}}|^{2}] &= {2}\int_{0}^{t} \|\jb{\nabla}^{-s}\pi_2 S(t-t')^{\ast}\phi\|_{L^2_{x}}^{2}dt' \\
& = {2}\int_{0}^{t} e^{-(t-t')}\bigg\| \frac{\sin((t-t')\jb{\nb})}{\jb{\nb} }\jb{\nb}^{-s}\phi_1 \\
& \hphantom{X + {2}\int_{0}^{t} e^{-(t-t')}\bigg\|} + \bigg(\cos((t-t')\jb{\nb})-\frac{\sin((t-t')\jb{\nb})}{2\jb{\nb}}\bigg)\jb{\nb}^{-s}\phi_2\bigg\|_{L^2_{x}}^{2}dt' \\
& \leq C(\|\phi_1\|_{H^{-1-s}}^{2}+\|\phi_2\|_{H^{-s}}^{2}),
\end{aligned}
\end{align*}
where the constant $C>0$ is independent of $t>0$. By setting $\phi_2\equiv 0$, we obtain \eqref{i1} for $\pi_1 \stick_{t}(\xi)$ which shows $\pi_1 \stick_{t}(\xi)\in C^{s-\ep}(\T^{2})$, for any fixed $t>0$. By setting $\phi_1\equiv 0$, we obtain \eqref{i1} for $\pi_2 \stick_{t}(\xi)$ which shows $\pi_2 \stick_{t}(\xi)\in C^{-1+s-\ep}(\T^{2})$, for any fixed $t>0$.
We now establish the temporal regularity of $\stick_{t}(\xi)$; namely, we verify \eqref{ii3}.
Let $0\leq t_1 \leq t_2$ and $\phi=\vec{\phi_1}{\phi_2}$ where $\phi_1, \phi_2$ are test functions with support in $B$.
We have
\begin{align}
\begin{split}
\jb{\stick_{t_2}(\xi)-\stick_{t_1}(\xi), \phi}_{L^2_x}=&\jb{\xi, \ind_{\{ t_1\leq t\leq t_2\}}\jb{\nb}^{-s}\pi_2 S(t_2-\cdot)^{\ast}\phi}_{L^2_{t,x}} \\
&+ \jb{\xi, \ind_{\{ 0\leq t\leq t_1\}}\jb{\nb}^{-s}[ \pi_2 S(t_2-\cdot)^{\ast}-\pi_2 S(t_1-\cdot)^{\ast}]\phi}_{L^2_{t,x}}.
\end{split}
\label{diff1}
\end{align}
Note that the test functions in the two terms on the right hand side of \eqref{diff1} have (essentially) disjoint supports in time. Thus, \eqref{whiteprop}, \eqref{diff1} and \eqref{derivTrade}
 imply 
\begin{align}\label{sticktdiff}
\begin{split} 
\E[ | \jb{ &\stick_{t_2}(\xi)-\stick_{t_1}(\xi), \phi}_{L^2_x}|^{2}] \\
& = \int_{t_1}^{t_2} \|\jb{\nb}^{-s}\pi_2 S(t_2-t)^{\ast}\phi\|_{L^2_x}^{2}dt \\
&\hphantom{XX}+ \int_{0}^{t_1} \| \jb{\nb}^{-s}[ \pi_2 S(t_2-t)^{\ast}-\pi_2 S(t_1-t)^{\ast}]\phi \|_{L^2_{x}}^{2}dt \\
& \les |t_2-t_1|(\|\phi_1\|_{H^{-1-s}}^{2}+\|\phi_2\|_{H^{-s}}^{2})+ |t_2-t_1|^{2\ta}( \|\phi_1\|_{H^{-1-s+\ta}}^{2}+\|\phi_2\|_{H^{-s+\ta}}^{2}) \\
& \les |t_2-t_1|^{2\ta}( \|\phi_1\|_{H^{-1-s+\ta}}^{2}+\|\phi_2\|_{H^{-s+\ta}}^{2}). 
\end{split}
\end{align}
Hence, putting $\phi_2=0$ (respectively, $\phi_1=0$), we obtain for $\ta \le \frac 12$,
\begin{align*}
\E[ | \jb{\pi_1\stick_{t_2}(\xi)-\pi_1\stick_{t_1}(\xi), \phi_1}|^{2}]&\les |t_2-t_1|^{2\ta}\|\phi_1\|_{H^{-1-s+\ta}}^{2}, \\
\E[ | \jb{\pi_2\stick_{t_2}(\xi)-\pi_2\stick_{t_1}(\xi), \phi_2}|^{2}]&\les |t_2-t_1|^{2\ta}\|\phi_2\|_{H^{-s+\ta}}^{2},
\end{align*}
which imply 
\begin{align*}
\pi_1 \stick_{t}(\xi) \in C([0,+\infty); C^{s-\ta-\ep}(\T^{2})) \quad \text{and} \quad \pi_2 \stick_{t}(\xi) \in C([0,+\infty); C^{-1-s-\ta-\ep}(\T^{2}))
\end{align*}
for any $0<\ta \leq \frac{1}{2}$.

\end{proof}

\begin{proposition} \label{PROP:stickZ}
Let $0<\al<s$. Then, 
\begin{align*}
\stick_{t}(\xi)\in C([0,+\infty); Z^{\al}) \subset C([0,+\infty); \X)
\end{align*}
 almost surely. 
In particular,
\begin{align*}
\sup_{\tau\geq 0}e^{\frac{\tau}{8}} \|S(\tau)\stick_{t}(\xi)\|_{\mathcal{C}^{\al}}<+\infty
\end{align*}
almost surely for every $t>0$ and for any $1\leq p<+\infty$. Moreover, for every $t>0$, 
\begin{align}
\E\big[   \|\stick_{t}(\xi)\|_{Z^{\al}}^{p}\big]^{\frac{1}{p}}\leq C_p, \label{momentsstick}
\end{align}
where $C_p>0$ is independent of $t$.
\end{proposition}

\begin{proof}
We begin with showing, for any fixed $t>0$, $\stick_{t}(\xi) \in Z^{\al}$. The argument here 
follows \cite[Proposition 2.2.3]{t19thesis}. Fix $t\geq 0$, $T\geq 1$ and let $t_1,t_2\in [T-1,T]$ such that $t_1\leq t_2$.
Then, we have 
\begin{align*}
\E \big[ \big| \jb{ S(t_2)\stick_{t}(\xi)-S(t_1)&\stick_{t}(\xi), \phi} \big|^2   \big] \\
&=2\E\bigg[  \bigg| \int_{0}^{t}\jb{\xi(t'), \jb{\nabla}^{-s}\pi_2( S(t+t_2-t')^{\ast}-S(t+t_1-t')^{\ast})\phi} dt'  \bigg|^{2}  \bigg] \\
& = \int_{0}^{t} \| \jb{\nabla}^{-s}\pi_2( S(t+t_2-t')^{\ast}-S(t+t_1-t')^{\ast})\phi\|_{L^2}^{2}dt' \\
& \les e^{-T}\int_{0}^{t} e^{-(t-t')}|t_2-t_1|^{2\ta}( \|\phi_{1}\|_{H^{-1-s+\ta}}^{2}+\|\phi_2\|_{H^{-s+\ta}}^{2})dt' \\
& \les e^{-T}|t_2-t_1|^{2\ta}( \|\phi_{1}\|_{H^{-1-s+\ta}}^{2}+\|\phi_2\|_{H^{-s+\ta}}^{2}).
\end{align*}
Given any $\phi=\begin{pmatrix}\phi_1\\ \phi_2\end{pmatrix}$, where $\phi_1$ and $\phi_2$ are test functions with support in a ball of unit radius, we have 
\begin{align*}
\E\big[ \big|  \jb{S(T-1)\stick_{t}(\xi), \phi}  \big|^{2}\big]& = \E\big[ \big| \jb{ \stick_{t}(\xi), S(T-1)^{\ast}\phi}\big|^{2}\big] \\
&=\| \pi_1 S(T-1)^{\ast}\phi\|_{H^{-1-s}}^{2}+\|\pi_2 S(T-1)^{\ast}\phi\|_{H^{-s}}^{2} \\
& \les e^{-T}\big( \| \phi_1\|_{H^{-1-s}}^{2}+\|\phi_2\|_{H^{-s}}^{2}).
\end{align*} 
Thus, by Proposition~\ref{PROP:REG}, we have 
\begin{align}
S(\tau)\stick_{t}(\xi)\in C_{\tau}([T-1,T]; \mathcal{C}^{\al}) \label{Ssstickcts}
\end{align}
for every $t\geq 0$. From \eqref{momentbd}, we have 
\begin{align}
\E\big[ \| S(\tau)\stick_{t}(\xi)\|_{C_{\tau}([T-1,T]; \mathcal{C}^{\al})}^{p}\big]^{\frac{1}{p}}\les e^{-\frac{T}{2}}\label{momentsstick1}
\end{align}
for $p$ sufficiently large, which by Chebyshev's inequality implies 
\begin{align*}
\prob  \big(  \| S(\tau)\stick_{t}(\xi)\|_{C_{\tau}([T-1,T]; \mathcal{C}^{\al})}>e^{-\frac{T}{8}} \big) \les e^{-\frac{T}{4}}.
\end{align*}
This is summable over $T\in \N$ so that by the Borel-Cantelli lemma, we have 
\begin{align*}
\limsup_{T\to \infty} e^{\frac{T}{8}}\|S(\tau)\stick_{t}(\xi)\|_{C_{\tau}([T-1,T];\mathcal{C}^{\al})}\leq 1.
\end{align*}
Then, \eqref{Ssstickcts} implies 
\begin{align*}
\sup_{T\in \N}e^{\frac{T}{8}}\|S(\tau)\stick_{t}(\xi)\|_{C_{\tau}([T-1,T];\mathcal{C}^{\al})} <+\infty
\end{align*}
and hence, $\stick_{t}(\xi)\in Z^{\al}$.
Furthermore, by \eqref{momentsstick1} and Minkowski's inequality, we have 
\begin{align*}
\E\big[   \|\stick_{t}(\xi)\|_{Z^{\al}}^{p}\big]^{\frac{1}{p}} & \leq \sum_{T\in \N}\E\bigg[  \bigg( \sup_{\tau\in [T-1,T]}e^{\frac{\tau}{8}}\|S(\tau)\stick_{t}(\xi)\|_{\CC^{\al}}  \bigg)^{p}   \bigg]^{\frac{1}{p}} \\
& \leq \sum_{T\in \N}Ce^{-\frac{3T}{8}}\leq C,
\end{align*}
which verifies \eqref{momentsstick}.
Now we show $\stick_{t}(\xi)\in C([0,+\infty);Z^{\al})$.
We fix $T\geq 1$, let $\tau_1,\tau_2\in [T-1,T]$ and $t_2\geq t_1\geq 0$.
Following the same lines as above and using the computations as in \eqref{sticktdiff}, we have 
\begin{align*}
&\phantom{=}\E\big[ \big|  \jb{S(T-1)(\stick_{t_2}(\xi)-\stick_{t_1}(\xi)), \phi}  \big|^{2}\big]\\
& = \E\big[ \big| \jb{ \stick_{t_2}(\xi)-\stick_{t_1}(\xi), S(T-1)^{\ast}\phi}\big|^{2}\big] \\
&\les |t_2-t_1|^{2\ta } \big(\| \pi_1 S(T-1)^{\ast}\phi\|_{H^{-1-s+\ta}}^{2}+\|\pi_2 S(T-1)^{\ast}\phi\|_{H^{-s+\ta}}^{2}\big) \\
& \les e^{-T}|t_2-t_1|^{2\ta}\big( \| \phi_1\|_{H^{-1-s+\ta}}^{2}+\|\phi_2\|_{H^{-s+\ta}}^{2}).
\end{align*}
Also, 
\begin{align*}
\E \big[ &\big| \jb{ S(\tau_2)(\stick_{t_2}(\xi)-\stick_{t_1}(\xi))-S(\tau_1)(\stick_{t_2}(\xi)-\stick_{t_1}(\xi)), \phi} \big|^2   \big] \\
& \les \int_{t_1}^{t_2}\big\| \jb{\nabla}^{-s}\pi_2 \big(S(\tau_2+t_2-t')-S(\tau_1+t_2-t') \big)^{\ast}\phi\big\|_{L^2_x}^{2}dt' \\
& \hphantom{Xx} +\int_{0}^{t_1}\big\| \jb{\nabla}^{-s}\pi_2 \big( S(\tau_2+t_2-t')-S(\tau_2+t_1-t')\\
&\hphantom{Xx+\int_{0}^{t_1}\big\| \jb{\nabla}^{-s}\pi_2 \big() } -S(\tau_1 +t_2-t')+S(\tau_1+t_1-t')  \big)^{\ast}\phi\big\|_{L_{x}^{2}}^{2}dt' \\
& =: I+II.
\end{align*}
By \eqref{derivTrade}, we have 
\begin{align*}
I \les e^{-T}|\tau_2-\tau_1|^{2\ta} \bigg( \int_{t_1}^{t_2}e^{-(t_2-t')}dt' \bigg)\|\phi\|_{\H^{-1-s+\ta}}^{2} \les e^{-T}|\tau_2-\tau_1|^{2\ta}|t_2-t_1|^{2\ta}  \|\phi\|_{\H^{-1-s+\ta}}^{2}.
\end{align*}
By similar computations but using the double mean value theorem (see \cite[Lemma 2.3]{DMVT}), we obtain 
\begin{align*}
II \les e^{-T}|\tau_2-\tau_1|^{2\ta}|t_2-t_1|^{2\ta}  \|\phi\|_{\H^{-1-s+2\ta}}^{2},
\end{align*}
and thus,
\begin{align*}
\E \big[ \big| \jb{ S(\tau_2)(\stick_{t_2}(\xi)-\stick_{t_1}(\xi))-S(\tau_1)(\stick_{t_2}(\xi)-\stick_{t_1}(\xi)),& \phi} \big|^2   \big]  \\
& \les e^{-T}|\tau_2-\tau_1|^{2\ta}|t_2-t_1|^{2\ta}  \|\phi\|_{\H^{-1-s+2\ta}}^{2}. 
\end{align*}
It follows from the moment bound \eqref{momentbd} that for $p\gg 1$, 
\begin{align}
\E \big[ \|S(\tau)(\stick_{t_2}(\xi)-\stick_{t_1}(\xi))\|_{C_{\tau}([T-1,T]; \mathcal{C}^{\al})}^{p} \big]^{\frac{1}{p}} \les |t_2-t_1|^{\ta}e^{-\frac{T}{4}}. \label{diffctyX}
\end{align}
Now, by Minkowski's inequality and \eqref{diffctyX}, we have
\begin{align*}
\E[ \| \stick_{t_2}(\xi)-\stick_{t_1}(\xi)\|_{Z^\al}^{p}]^{\frac{1}{p}} & \leq \sum_{T\in \N} \E\bigg[ \bigg(\sup_{\tau\in [T-1,T]}e^{\frac{\tau}{8}}\|S(\tau)(\stick_{t_2}(\xi)-\stick_{t_1}(\xi))\|_{\mathcal{C}^{\al}}\bigg)^{p}\bigg]^{\frac{1}{p}} \\
& \les \sum_{T\in \N} e^{\frac{T}{8}} \E \big[ \|S(\tau)(\stick_{t_2}(\xi)-\stick_{t_1}(\xi))\|_{C_{\tau}([T-1,T]; \mathcal{C}^{\al})}^{p} \big]^{\frac{1}{p}} \\
& \les |t_2-t_1|^{\ta}\sum_{T\in \N} e^{-\frac{3T}{8}} \les  |t_2-t_1|^{\ta}.
\end{align*}
Now, Kolmogorov's continuity criterion implies $\stick_{t}(\xi)\in C([0,+\infty);Z^{\al})$.
In particular, we have 
\begin{align}
\E\big[ \|\stick_{t}(\xi)\|_{C([0,T];Z^{\al})}^{p}\big]^{\frac{1}{p}}\le C(p,T). \label{stickC01}
\end{align}

\end{proof}

\noi
Throughout the rest of this article, we fix the value of $\al\in (0, s\wedge \frac{1}{3})$.

\section{Well-posedness theory for (SDNLW)}

\subsection{Local well-posedness}
We have the following local well-posedness result for the equation~\eqref{SDNLW}.
\begin{proposition}\label{PROP:LWP}
Let $N\in \N$, $s>0$, $\g \in \R$ and $\u_0 \in \X$. Then, the equations 
\begin{align}
\v(t) &=S(t)\v_0 - \int_0^t S(t-t')\vec{0}{\NN_{\g}[ \pi_1 (S(t')\u_0+\stick_{t'}(\xi)+\v(t'))]  } dt', \label{veqnd}\\
\v_{N}(t) &=S(t)\v_0 - \int_0^t S(t-t') P_{\leq N} \vec{0}{\NN_{\g}[\pi_1 P_{\leq N}(S(t')\u_0+\stick_{t'}(\xi)+\v_{N}(t'))] } dt'.
\label{veqnNd}
\end{align}
are almost surely locally-well posed in $\H^1(\T^2)$ in the following sense: For every $\v_0\in \H^1$, there exists $T(\|\v_0\|_{\H^1}, \|\u_0\|_{\X}, \|\stick_{t}(\xi)\|_{C([0,1];\W^{\al,6})})>0$ such that there exists a unique solution $\v_{N}(t; \v_0, \u_0,\xi)\in C_{t}\H^1$ and $\v(t;\v_0, \u_0,\xi)\in C_{t}\H^{1}$, to \eqref{veqnNd} and \eqref{veqnd}, respectively on a maximal time interval $[0,T^{\ast}(\v_0,\u_0, \xi)), [0, T^{\ast}_{N}(\v_0,\u_0,\xi))$ respectively, with $T^{\ast}$, $T^{\ast}_{N}>T$ almost surely. 
Moreover, we have the following blow-up criterion: if $T^{\ast}<+\infty$, then 
\begin{align*}
\lim_{t\to T^{\ast}}\|\v(t)\|_{\H^1}=+\infty, 
\end{align*}
or if $T^{\ast}_{N}<+\infty$, then
\begin{align*}
 \lim_{t\to T_{N}^{\ast}}\|\v_N(t)\|_{\H^1}=+\infty.
\end{align*}
\end{proposition}

\begin{proof}
We let 
\begin{align*}
\Gb[\v(t)]&:= S(t)\v_0 - \int_0^t S(t-t')\vec{0}{v(t')^3+Q(v(t'))  } dt', \\
\Gb_{N}[\v_N(t)]&:=S(t)\v_0 - \int_0^t S(t-t')P_{\leq N} \vec{0}{(P_{\leq N}v_{N}(t'))^3+Q_{N}(P_{\leq N}v_N (t')) } dt',
\end{align*}
where $Q_{N}$ is the quadratic polynomial
\begin{align}
Q_{N}(x)=a_N x^{2}+b_N x+c_N, 
\label{QN}
\end{align}
and we define
\begin{align}
a_N& =a_N(t)= 3u_{0,N}+3\pi_{1}\stick_{t}^{N}(\xi), \label{aN}\\
b_N& =b_N(t)=3( u_{0,N}^2+2 u_{0,N}\pi_1 \stick_{t}^{N}(\xi)+\wick{(\stick_{t}^{N}(\xi))^{2}}_{\g} ), \label{bN}\\
c_N&=c_N(t)=u_{0,N}^{3}+3u_{0,N}^{2}\pi_1\stick_{t}^{N}(\xi)+3u_{0,N}\wick{(\stick_{t}^{N}(\xi))^{2}}_{\g} + \wick{(\stick_{t}^{N}(\xi))^{3}}_{\g}, \label{cN}
\end{align}
where we have also defined
\begin{align}
u_{0,N}&=\pi_1  P_{\leq N}(S(t)\u_0),  \notag\\
\stick_{t}^{N}(\xi)& =\pi_{1} P_{\leq N}\stick_{t}(\xi), \notag\\
\wick{(\stick_{t}^{N}(\xi))^{2}}_{\g} &= (\pi_1 P_{\leq N}\stick_{t}(\xi))^{2}-\g, \notag \\
\wick{(\stick_{t}^{N}(\xi))^{3}}_{\g} & = (\pi_1 P_{\leq N}\stick_{t}(\xi))^{3}-3\g \pi_1 P_{\leq N}\stick_{t}(\xi). \notag
\end{align} 
The above definitions clearly extend to the case when $N=+\infty$ in view of the notation $P_{\leq +\infty}=\text{Id}$.

From the boundedness of $P_{\leq N}$ and Young's inequality, for any $t\in [0,1]$, we have
\begin{align}
\|a_N(t)\|_{L^{6}_{x}}& \les \|\u_0\|_{\X}+\| \stick_{t}^{N}(\xi)\|_{\W^{\al,6}}, \label{aest}\\
\| b_N(t)\|_{L^3_{x}}&\les \|\u_0\|_{\X}^{2}+\|\stick_{t}^{N}(\xi)\|_{\W^{\al,6}}^{2}+\| \wick{(\stick_{t}^{N}(\xi))^{2}}_{\g}\|_{L^3_x},\label{best} \\
\|c_N(t)\|_{L^2_{x}}&\les \|\u_0\|^{3}_{\X}+\|\stick_{t}^{N}(\xi)\|_{\W^{\al,6}}^{3}+\| \wick{(\stick_{t}^{N}(\xi))^{2}}_{\g}\|_{L^3_x} + \| \wick{(\stick_{t}^{N}(\xi))^{3}}_{\g}\|_{L^2_x}. \label{cest}
\end{align}
We will show $\Gb$ and $\Gb_{N}$ are strict contractions on a ball 
\begin{align*}
B_{R}=\{ \u \, :\, \|\u\|_{C([0,T];\H^1)}\leq R\} \subseteq C([0,T];\H^1)
\end{align*}
for a time $T(\|\v_0\|_{\H^1}, \|\u_0\|_{\X}, \| \stick_{t}(\xi)\|_{C([0,1];\W^{\al,6})})$ which is positive almost surely. 
For $\v \in B_R$, the boundedness of $S(t)$ on $\H^1$, $P_{\leq N}$ on $\H^1$ and $L^{3}$, Sobolev embeddings and \eqref{aest}, \eqref{best} and \eqref{cest} imply
\begin{align*}
&\phantom{\les\ }\| \Gb[\v_N (t)]\|_{\H^1} \\
& \les \|\v_0\|_{\H^1}+ T\sup_{0\leq t\leq T} \Bigg\| \vec{0}{(P_{\leq N}v_{N}(t'))^3+Q_{N}(P_{\leq N}v_N (t')) } \Bigg\|_{\H^1} \\
& \les \|\v_0\|_{\H^1}+ T\sup_{0\leq t\leq T} \| (P_{\leq N}v_{N}(t'))^3+Q_{N}(P_{\leq N}v_N (t')) \|_{L^2} \\
& \les \|\v_0\|_{\H^1}+ T\sup_{0\leq t\leq T} \bigg( \|\v_N\|_{\H^1}^3+\|a_N\|_{L^6}\|\v_N\|_{\H^1}^{2}+\|b_N\|_{L^3}\|\v_N\|_{\H^1}+\|c_N\|_{L^2}   \bigg) \\
& \les \|\v_0\|_{\H^1}+T\sup_{0\leq t\leq T} \bigg( \|a_N\|_{L^6}^{3}+\|b_N\|_{L^3}^{\frac{3}{2}}+\|c_N\|_{L^2}   \bigg) +TR^3 \\
&\les \|\v_0\|_{\H^1}+T( \|\u_0\|_{\X}^3 + \| \stick_{t}(\xi)\|^{3}_{C([0,1]; \W^{\al,6})}+\| \wick{(\stick_{t}^{N}(\xi))^{2}}_{\g}\|_{C([0,1]; L^3)}^{\frac{3}{2}} \\
& \hphantom{XXXXX}+ \| \wick{(\stick_{t}^{N}(\xi))^{3}}_{\g}\|_{C([0,1]; L^2)})+TR^3.
\end{align*}
Thus, choosing $R=2C(1+\|\v_0\|_{\H^{1}})$ and 
\begin{align*}
T< 1 \wedge {R}\big[2C\big(& \|\u_0\|_{\X}^3 + \| \stick_{t}(\xi)\|^{3}_{C([0,1]; \W^{\al,6})}\\
&+\| \wick{(\stick_{t}^{N}(\xi))^{2}}_{\g}\|_{C([0,1]; L^3)}^{\frac{3}{2}} + \| \wick{(\stick_{t}^{N}(\xi))^{3}}_{\g}\|_{C([0,1]; L^2)} +R^3\big)\big]^{-1},
\end{align*} 
shows $\Gb$ maps $B_R$ into itself. 
For $\v_N, \w_N\in B_R$ with $\v_N (0)=\w_N(0)=\v_0$, \eqref{aest} and \eqref{best} yield
\begin{align*}
\| &\Gb[\v_N (t)] - \Gb[\w_N (t)] \|_{\H^1}  \\ 
& \les T\sup_{0\leq t\leq T} \Bigg\| \vec{0}{(P_{\leq N}v_{N}(t'))^3+Q_{N}(P_{\leq N}v_N (t'))- (P_{\leq N}w_{N}(t'))^3-Q_{N}(P_{\leq N}w_N (t'))} \Bigg\|_{\H^1} \\
& \les T\sup_{0\leq t\leq T} \|(P_{\leq N}v_{N}(t'))^3+Q_{N}(P_{\leq N}v_N (t'))- (P_{\leq N}w_{N}(t'))^3-Q_{N}(P_{\leq N}w_N (t'))\|_{L^2} \\
& \les T\sup_{0\leq t\leq T} ( \|\v_N\|_{\H^1}^{2}+\|\w_N\|_{\H^1}^2+\|a_N\|_{L^6}(\|\v_N\|_{\H^1}+\|\w_N\|_{\H^1})+\|b_N\|_{L^3}) \|\v_N -\w_N\|_{\H^1} \\
&\les  T (R^2+\|\u_0\|_{\X}^2+\| \stick_{t}(\xi)\|_{C([0,1];\W^{\al,6})}^{2}+\| \wick{(\stick_{t}^{N}(\xi))^{2}}_{\g}\|_{C([0,1]; L^3)})\|\v_N-\w_N\|_{C([0,T];\H^1)}.
\end{align*} 
Thus reducing $T$ further, if necessary, we have 
\begin{align*}
\| \Gb[\v_N (t)] - \Gb[\w_N (t)] \|_{C([0,T];\H^1)}\leq \frac{1}{2}\|\v_N-\w_N\|_{C([0,T];\H^1)}. 
\end{align*}
\end{proof}

\begin{remark}
Following the argument in \cite{GKO}, we expect to be able to extend the local well-posedness statement to $\v \in \H^{\sigma}$, $\sigma > \frac 14$. However, in view of the global well-posedness result given by Proposition \ref{PROP:GWP}, we decided not to pursue this improvement, since we will always work with $\v \in \H^1$.
\end{remark}

We also have the following continuity and convergence results, the proofs of which follow from Proposition \ref{PROP:LWP} and standard arguments. 
See for instance the proofs of Proposition 3.2 and Lemma 3.3 in \cite{t18erg}.

\begin{lemma}\label{LEM:ctyofflow} Given $\u_0\in \X$, let $\v(\u_0)$ denote the solution to \eqref{veqn} on the time interval $[0, T^{\ast})$ and let $T<T^{\ast}$. Then, there exists a neighbourhood $\mathcal{U}\subset \X$ of $\u_0$ such that for every $\tilde{\u}_{0}\in \mathcal{U}$, $\v(\tilde{\u}_{0})$ is the solution to \eqref{veqn} on $[0,T]$ and we have 
\begin{align*}
\lim_{\|\tilde{\u}_{0}-\u_0 \|_{\X}\to 0}\big\| \v(\tilde{\u}_{0})-\v(\u_0)\big\|_{C([0,T]; \H^{1})} =0.
\end{align*}

\end{lemma}

\begin{lemma} \label{LEM:APPROXV}
Let $\v_0\in \H^1$, $\u_0\in \X$ and $\v_N$ the solution to \eqref{veqnNd}. Suppose 
\begin{align*}
\sup_{N\in \N}\sup_{0\leq t\leq \overline{T}}\|\v_N(t)\|_{\H^1}\leq K<+\infty.
\end{align*}
Then, the solution $\v$ to \eqref{veqnd} satisfies $T^{\ast}\geq \overline{T}$, 
\begin{align*}
\sup_{0\leq t\leq \overline{T}}\|\v(t)\|_{\H^1}\leq K,
\end{align*}
and
\begin{align*}
\lim_{N\to \infty} \| \v-\v_N \|_{C([0,\overline{T}];\H^1)}=0.
\end{align*}

\end{lemma}

\subsection{Global well-posedness}

For fixed $N\in \N\cup \{0\}$, let $\v_N$ be a solution to \eqref{veqnN} and write $\v_N=\vec{v_N}{\dt v_N}$.
The goal in this section is to prove the following result:

\begin{proposition}\label{PROP:GWP}
Let $N\in \N$ and $\v_{N}$ solve \eqref{veqnN}. Then, for $0<\al<s \wedge \frac{1}{3}$, we have 
\begin{align*}
\|\v_N(t)\|_{\H^1}^{2} &\les \eta(t; \|\u_0\|_{\X}) \bigg(1+\|P_{\leq N}\u_0\|_{\X}^{13}+\|\pi_1 P_{\leq N}\stick_{t}\|_{L^6}^{6} \\
& \hphantom{XX}+ \int_0^{t} e^{-c(t-t')} \big(\| P_{\leq N}\stick_{t'}\|_{\W^{\al,\frac{4}{\al}}}^{\frac{8}{\al}}+\| \wick{(P_{\leq N}\stick_{t'})^{2}}_{\g}\|_{L^4}^{7}+\| \wick{(P_{\leq N}\stick_{t'})^{3}}_{\g}\|_{L^2}^{2}\big) dt' \bigg),
\end{align*}
which is finite almost surely, where for some $c > 0$,
\begin{align*}
\eta(t;\|\u_0\|_{\X}):=e^{32C\|\u_0\|_{\X}}\ind_{\{t\leq \tau\}}+(e^{32C\|\u_0\|_{\X} - ct} + 1)\ind_{\{t>\tau\}},
\end{align*}
and $\tau=\tau(\|\u_0\|_{\X}):=8\log \big( 20C\|\u_0\|_{\X}\vee 10   \big)$.
 Furthermore, denoting $\v$ the solution to \eqref{veqn}, then $\v$ exists on $[0,\infty)$, and satisfies 
\begin{align*}
\begin{split}
\|\v(t)\|_{\H^1}^{2} &\les  \eta(t; \|\u_0\|_{\X})\bigg( 1+\|\u_0\|_{\X}^{13}+\|\pi_1 \stick_{t}(\xi)\|_{L^6}^{6} \\
& \hphantom{XX}+ \int_0^{t} e^{-c(t-t')}\big(\| \stick_{t'}(\xi)\|_{\W^{\al,\frac{4}{\al}}}^{\frac{8}{\al}}+\| \wick{\stick_{t'}(\xi)^{2}}_{\g}\|_{L^4}^{7}+\| \wick{\stick_{t'}(\xi)^{3}}_{\g}\|_{L^2}^{2}\big) dt' \bigg),
\end{split}
\end{align*}
and for every $T<\infty$, we have 
\begin{align*}
\lim_{N\to \infty} \|\v-\v_N\|_{C([0,T];\H^1)}=0 \,\,\textup{a.s.}
\end{align*}
In particular, for every $p < \infty$, there exists $C_p > 0$ such that 
\begin{equation}\label{vmoments}
\big(\E \big[\|\v(t)\|_{\H^1}^p\big]\big)^\frac 1 p \le \eta(t; \|\u_0\|_{\X})^\frac 12(C_p + \| \u_0\|_{\X}^{7}).
\end{equation}
\end{proposition}

 We define the energy 
\begin{align}
\begin{split}
E(\v_N):= &\frac{1}{2}\int (\dt P_{\leq N}v_N)^{2} +\frac{1}{2}\int (P_{\leq N}v_N)^2 +\frac{1}{2}\int |\nb P_{\leq N}v_N|^{2} \\
& +\frac{1}{4}\int (P_{\leq N}v_N)^4 +\frac{1}{8}\int (P_{\leq N}v_N+\dt P_{\leq N}v_N)^{2}.
\end{split}
\label{energy}
\end{align}
The brunt of the work will be to obtain the following energy estimate.

\begin{proposition}\label{PROP:EE}
Let $0<\al <s\wedge \frac{1}{3}$. Then, there exists $c>0$ such that for $\v_N$ a solution to \eqref{veqnN}, we have 
\begin{align*}
&~E(\v_N(t))\\
&\les \eta(t; \|\u_0\|_{\X}) \bigg(1+\|P_{\leq N}\u_0\|_{\X}^{13}+\|\pi_1 \stick_{t}^{N}(\xi)\|_{L^6}^{6} \\
& \phantom{XXXXXXXXX}+ \int_0^{t} e^{-c(t-t')}\big(\| \stick^{N}_{t}(\xi)\|_{\W^{\al,\frac{4}{\al}}}^{\frac{8}{\al}}+\| \wick{(\stick_{t}^{N}(\xi))^{2}}_{\g}\|_{L^4}^{7}+\| \wick{(\stick_{t}^{N}(\xi))^{3}}_{\g}\|_{L^2}^{2}\big) dt'\bigg).
\end{align*}
\end{proposition} 

In view of \eqref{energy} and Lemma~\ref{LEM:APPROXV}, Proposition~\ref{PROP:GWP} now follows from Proposition~\ref{PROP:EE}.

\noi
Let $\w_N :=P_{\leq N}\v_N$ and note $E(\v_N)=E(\w_N)$. Furthermore, $\w_N$ satisfies 
\begin{align}
\w_{N}(t) =S(t)P_{\leq N}\v_0 - \int_0^t S(t-t')P_{\leq N} \vec{0}{w_{N}^3 + Q_{N}(w_N)} dt',
\label{weqngwp}
\end{align}
where $w_N := \pi_1 \w_N$ and $Q_N$ is the quadratic polynomial in \eqref{QN}, with coefficients given by \eqref{aN}, \eqref{bN}, \eqref{cN}.
Clearly, $\w_N$ is smooth in space and time and hence $E(\w_N)$ is differentiable. Then, using \eqref{weqngwp}, we compute 
\begin{align}
\dt E(&\w_N) = \int (\dt^{2}w_N+w_N-\Delta w_N+w_N^{3})\dt w_N +\frac{1}{4}\int (\dt^2 w_N+\dt w_N)(w_N+\dt w_N) \notag \\
 &= -\int (\dt w_N)^2-\int (\dt w_N)P_{\leq N} Q_N(w_N)+\int P_{>N}(w_N^3)\dt w_N +\frac{1}{8}\dt \bigg( \int (\dt w_N)^2  \bigg) \notag \\ 
 &\hphantom{X}+\frac{1}{4}\int (-w_N+\Delta w_N-P_{\leq N}(w_N^3)-P_{\leq N}Q_N(w_N))w_N +\frac{1}{4}\int (\dt w_N)^2  \notag\\
 & = -\frac{1}{4}\bigg( 3\int (\dt w_N)^2+\int w_N^2 +|\nb w_N|^2+w_N^4 \bigg)  \label{e1}  \\
 &\hphantom{X}+\frac{1}{8}\dt \bigg( \int (\dt w_N)^2  \bigg) \label{e2} \\
 &\hphantom{X}+\int P_{>N}(w_N^3)\bigg(\frac{1}{4}w_N+\dt w_N\bigg) \label{e3} \\
 & \hphantom{X}-\frac{1}{4}\int w_N P_{\leq N}Q_N(w_N) \label{e4}\\
 & \hphantom{X} -\int \dt w_N P_{\leq N}Q_N(w_N). \label{e5}
\end{align}

\noi
Since $P_{\leq N}\w_N=\w_N$, we may remove the projection $P_{\leq N}$ in \eqref{e4} and \eqref{e5}.
We consider a few of the terms above separately. 

\begin{lemma}\label{LEM:E1}
If $\w_N$ solves \eqref{weqngwp}, then 
\begin{align*}
\eqref{e3}=0.
\end{align*}
\end{lemma}
\begin{proof}
By definition, $P_{>N}\w_N=0$. Hence, 
\begin{align*}
\eqref{e3}=\int w_N^3 P_{>N} \bigg(\frac{1}{4}w_N+\dt w_N\bigg)=0.
\end{align*}
\end{proof}

\begin{lemma}\label{LEM:E2}
For any $0<\al<s\wedge \frac 13$, we have 
\begin{align*}
\bigg\vert \int w_N Q_N(w_N) \bigg\vert \les &  ( \|P_{\leq N}\u_0\|_{X^{\al}}+\|\stick_{t}^{N}(\xi)\|_{\W^{\al,4}})E(\w_N)^{\frac 34}\\
& \hphantom{X}+ (\|P_{\leq N}\u_0\|_{\X}^{2}+\|\stick_{t}^{N}(\xi)\|_{\W^{\al,6}}^{2}+\|\wick{(\stick_{t}^{N}(\xi))^{2}}_{\g}\|_{L^2})E^{\frac 12}(\w_N) \\
& \hphantom{X}+(\|P_{\leq N}\u_0\|_{\X}^{3}+\|\stick_{t}^{N}(\xi)\|_{\W^{\al,6}}^{3}+\|\wick{(\stick_{t}^{N}(\xi))^{2}}_{\g} \|_{L^4}^{\frac{3}{2}} \\
& \hphantom{X}+  \|\wick{(\stick_{t}^{N}(\xi))^{3}}_{\g}\|_{L^{2}}) E(\w_N)^{\frac 12}.
\end{align*}
\end{lemma}

\begin{proof}
From \eqref{QN} and H\"older's inequality, we have
\begin{align*}
\bigg\vert \int w_N Q_N(w_N) \bigg\vert \les \|a_N\|_{L^4}E^{\frac 34}(\w_N)+(\|b_N\|_{L^2}+\|c_N\|_{L^2})E^{\frac 12}(\w_N).
\end{align*}
By \eqref{aN}, we have
\begin{align*}
\|a_N\|_{L^4} \les \| P_{\leq N}\u_0\|_{\X} + \| \pi_1 P_{\leq N} \stick_{t}(\xi)\|_{L^{4}}.
\end{align*}
Similarly, \eqref{bN} and H\"older's inequality imply
\begin{align*}
\|b_N\|_{L^2}\les \|\pi_1 P_{\leq N}S(t)\u_0\|_{L^{4}}^{2}+\|\pi_1 \stick_{t}^{N}(\xi)\|_{L^4}^{2}+ \| \wick{(\stick_{t}^{N}(\xi))^{2}}_{\g}\|_{L^2}.
\end{align*}
Recalling the formula \eqref{cN} for $c_N$, by H\"older's and Young's inequalities, we obtain
\begin{align}
\begin{split}
\|c_N\|_{L^2} \les & \|P_{\leq N}\u_0\|_{\X}^{3}+\|\stick_{t}^{N}(\xi)\|_{\W^{\al,6}}^{3}+\|\wick{(\stick_{t}^{N}(\xi))^{2}}_{\g}\|_{L^4}^{\frac{3}{2}} \\
& + \|\wick{(\stick_{t}^{N}(\xi))^{3}}_{\g}\|_{L^{2}}.
\end{split}
\label{cNL2}
\end{align}
\end{proof}

Using \eqref{QN}, we write 
\begin{align}
\eqref{e5} = -\int a_N w_N^2 \dt w_N -\int b_N w_N \dt w_N +\int c_N \dt w_N. 
\label{intbyparts0}
\end{align}
However, we cannot directly estimate the first term. To get around this, we employ an integration by parts trick; namely, we write 
\begin{align}
-\int a_N w_N^2 \dt w_N = -\frac{1}{3} \int a_N \dt ( w_N^3) =-\frac{1}{3}\dt \bigg( \int a_N w_N^3  \bigg)
+\frac{1}{3}\int (\dt a_N) w_N^{3}. \label{intbyparts}
\end{align}

\begin{lemma}\label{LEM:E5}
For every $0<\al<s\wedge \frac 13$, we have
\begin{align*}
\bigg\vert  \int b_N w_N \dt w_N  \bigg\vert  & \les (\|P_{\leq N}\u_0\|_{\X}^{2}+\|\stick_{t}^{N}(\xi)\|_{\W^{\al,6}}^{2}+ \| \wick{(\stick_{t}^{N}(\xi))^{2}}_{\g}\|_{L^3})E^{\frac{5}{6}+}(\w_N),   \\
\bigg\vert \int c_N \dt w_N \bigg\vert & \les (\|P_{\leq N}\u_0\|_{\X}^{3}+\|\stick_{t}^{N}(\xi)\|_{\W^{\al,6}}^{3}+\|\wick{(\stick_{t}^{N}(\xi))^{2}}_{\g}\|_{L^4}^{\frac{3}{2}} \\
&\hphantom{XXX} +   \|\wick{(\stick_{t}^{N}(\xi))^{3}}_{\g}\|_{L^{2}})E(\w_N)^{\frac 12}, \\
\bigg\vert \int (\dt a_N) w_N^3 \bigg\vert & \les e^{-\frac{t}{8}}\|P_{\leq N}\u_0\|_{\X}E(\w_N)+\|\stick_{t}^{N}(\xi)\|_{\W^{\al,\frac{4}{\al}}}E^{1-\frac{\al}{4}}(\w_N),\\
\bigg\vert \int a_N w_N^3 \bigg\vert & \les (\|P_{\leq N}\u_0\|_{\X}+\|\stick_{t}^{N}(\xi)\|_{\W^{\al,6}})E^{\frac{5}{6}}(\w_N).
\end{align*}
\end{lemma}

\begin{proof}
The second estimate follows from Cauchy-Schwarz and \eqref{cNL2}.
By interpolation $( \frac{1}{6}=\frac{1-\ta}{4}+\frac{\ta}{q}$) and Sobolev embedding, for $4\ll q<\infty$, we have
\begin{align}
\|w_N\|_{L^{6}} \leq \|w_N\|_{L^4}^{1-\ta}\|w_N\|_{L^q}^{\ta} \les E^{\frac{1-\ta}{4}+\frac{\ta}{2}}(\w_N)= E^{\frac{1}{3}+\frac{1}{3(q-4)}}(\w_N). \label{L6wN}
\end{align}
Thus, by H\"{o}lder's inequality and \eqref{L6wN},
\begin{align*}
\bigg\vert  \int b_N w_N \dt w_N  \bigg\vert & \leq \|b_N\|_{L^3}\|w_N\|_{L^6}\|\dt w_N\|_{L^2} \les \|b_N\|_{L^3}E^{\frac{5}{6}+\frac{1}{3(q-4)}}(\w_N).
\end{align*}
By choosing $q$ big enough, we obtain the first inequality.

Clearly, 
\begin{align*}
\| w_{N}^{3}\|_{W^{1,1}}& \les \|w_N\|_{H^{1}}\|w_{N}\|_{L^{4 }}^{2} \les  E(\w_{N}),
\\
\| w_{N}^{3}\|_{L^{\frac{4}{3}}}& \les E^{\frac{3}{4}}(\w_{N}),
\end{align*}
so by Gagliardo-Nirenberg
\begin{align}
\norm{w_N^3}_{W^{1-\al, \left(1-\frac{\al}{4}\right)^{-1}}} \les \norm{w_N^3}_{L^{\frac43}}^{\al} \norm{w_N^3}_{W^{1,1}}^{1-\al} \les E(\w_N)^{1-\frac \al 4}, \label{GN2}
\end{align}
Hence, duality and \eqref{GN2} give
\begin{align*}
\begin{split}
\bigg\vert \int (\dt \stick_t^N(\xi)) w_N^3 \bigg\vert & \les \|\dt \stick_t^N(\xi)\|_{W^{-1+\al,\frac{4}{\al}}}\|w_{N}^3\|_{W^{1-\al,\left(1-\frac{\al}{4}\right)^{-1}}} \\
&\les \|\stick_{t}^{N}(\xi)\|_{\W^{\al,\frac{4}{\al}}} E^{1-\frac{\al}{4}}(\w_N).
\end{split}
\end{align*}
By fractional Leibniz (Lemma \ref{fraclieb}) and Sobolev embeddings, we have
\begin{align*}
\| w_{N}^{3}\|_{W^{1-\alpha,\left(1 - \frac \al 2\right)^{-1}}}\les \|w_N\|_{W^{1-\alpha, \frac 2{1-\alpha}}}\|w_{N}\|_{L^{4 }}^{2} \les \|w_N\|_{H^{1}}\|w_{N}\|_{L^{4 }}^{2} \les  E(\w_{N}).
\end{align*}
and hence 
\begin{align*}
\bigg\vert \int \dt(\pi_1 P_{\leq N}S(t)\u_0) w_N^3  \bigg\vert & = \bigg\vert \int (\pi_2 P_{\leq N}S(t)\u_0) w_N^3  \bigg\vert  \\
& \leq \|\pi_2 P_{\leq N}S(t)\u_0\|_{W^{-1+\al, \frac 2 \al}}\|w_N^3\|_{W^{1-\al, \left(1 - \frac \al2\right)^{-1}}} \\
& \les e^{-\frac{t}{8}}\| P_{\leq N}\u_0\|_{X^{\al}}E(\w_N).
\end{align*}
Finally, by H\"{o}lder's inequality and interpolation, we have 
\begin{align*}
\bigg\vert \int a_N w_N^3 \bigg\vert & \leq \|a_N\|_{L^6}\|w_{N}\|^{3}_{L^{\frac{18}{5}}} \les \|a_N\|_{L^6}\|w_N\|_{L^2}^{\frac{1}{3}}\|w_N\|_{L^4}^{\frac{8}{3}}  \les \|a_N\|_{L^6}E^{\frac{5}{6}}(\w_N).
\end{align*}
\end{proof}

\begin{proof}[Proof of Proposition~\ref{PROP:EE}]

Define 
\begin{align*}
F(\w_N)= E(\w_N)-\frac{1}{8}\int (\dt w_N)^2 +\frac{1}{3}\int a_N w_N^3,
\end{align*}
and from \eqref{energy} and Lemma~\ref{LEM:E5}, there exists a constant $C>0$ such that 
\begin{align}
F(\w_N)& \leq \frac{5}{4}E(\w_N)+C(\|P_{\leq N}\u_0\|_{\X}+\|\pi_1 \stick_{t}^{N}(\xi)\|_{L^6})^{6}, \label{FbyE}\\
E(\w_N) & \leq 2F(\w_N)+2C(\|P_{\leq N}\u_0\|_{\X}+\|\pi_1 \stick_{t}^{N}(\xi)\|_{L^6})^{6}. \label{EbyF}
\end{align}
From \eqref{e1}, \eqref{e2}, \eqref{e3}, \eqref{e4}, \eqref{e5}, Lemma~\ref{LEM:E1}, \eqref{intbyparts0} and \eqref{intbyparts}, 
we have 
\begin{align*}
\dt F(\w_N) & = -\frac{1}{4}\bigg( 3\int (\dt w_N)^2+\int w_N^2 +|\nb w_N|^2+w_N^4 \bigg)    \\
 & \hphantom{X}-\frac{1}{4}\int w_N P_{\leq N}Q_N(w_N) \\
 & \hphantom{X} +\frac{1}{3}\int (\dt a_N)w_N^3 -\int b_N w_N \dt w_N + \int c_N \dt w_N.
\end{align*}
Thus, by Lemma~\ref{LEM:E2}, Lemma~\ref{LEM:E5}, \eqref{FbyE} and Young's inequality, we have (where the constant $C>0$ can  change from line to line)
\begin{align*}
\dt F(\w_N) & \leq -\frac{1}{2}E(\w_N)+ CE^{\frac{3}{4}}(\w_N)(\|P_{\leq N}\u_0\|_{\X}+\| \stick_{t}^{N}(\xi)\|_{\W^{\al,4}}) \\
& \hphantom{X} + CE^{\frac 12}(\w_N)(\|P_{\leq N}\u_0\|_{\X}^2+\| \stick_{t}^{N}(\xi)\|_{\W^{\al,6}}^2+\| \wick{(\stick_{t}^{N}(\xi))^{2}}_{\g}\|_{L^2}) \\
& \hphantom{X} + CE^{\frac{1}{2}}(\w_N)(\|P_{\leq N}\u_0\|_{\X}^3+\| \stick_{t}^{N}(\xi)\|_{\W^{\al,6}}^3+\| \wick{(\stick_{t}^{N}(\xi))^{2}}_{\g}\|_{L^4}^{\frac{3}{2}}+\| \wick{(\stick_{t}^{N}(\xi))^{3}}_{\g}\|_{L^2}) \\
& \hphantom{X} + CE^{1-\frac{\al}{4}}(\w_N)\| \stick_{t}^{N}(\xi)\|_{\W^{\al,\frac{4}{\al}}}+CE(\w_N)e^{-\frac{t}{8}}\|P_{\leq N}\u_0\|_{\X} \\
& \hphantom{X}+CE^{\frac{5}{6}+}(\w_N)(\|P_{\leq N}\u_0\|_{\X}^2+\| \stick_{t}^{N}(\xi)\|_{\W^{\al,6}}^2+\| \wick{(\stick_{t}^{N}(\xi))^{2}}_{\g}\|_{L^3}) \\
& \hphantom{X} + CE^{\frac{1}{2}}(\w_N)(\|P_{\leq N}\u_0\|_{\X}^3+\| \stick_{t}^{N}(\xi)\|_{\W^{\al,6}}^3+\| \wick{(\stick_{t}^{N}(\xi))^{2}}_{\g}\|_{L^4}^{\frac{3}{2}}+\| \wick{(\stick_{t}^{N}(\xi))^{3}}_{\g}\|_{L^2}) \\
& \leq -\frac{1}{2}E(\w_N)+\frac{1}{4}E(\w_N)+Ce^{-\frac{t}{8}}\|P_{\leq N}\u_0\|_{\X}E(\w_N)\\
&\hphantom{X}+C(1+\|P_{\leq N}\u_0\|_{\X}^{13}+\| \stick_{t}^{N}(\xi)\|_{\W^{\al,\frac{4}{\al}}}^{\frac{4}{\al}}+\| \wick{(\stick_{t}^{N}(\xi))^{2}}_{\g}\|_{L^4}^{7} +\| \wick{(\stick_{t}^{N}(\xi))^{3}}_{\g}\|_{L^2}^{2})  \\
& \leq -\frac{1}{5}F(\w_N)+2Ce^{-\frac{t}{8}}\| P_{\leq N}\u_0\|_{\X}F(\w_N) \\
& \hphantom{X} +C(1+2e^{-\frac{t}{8}}\| P_{\leq N}\u_0\|_{\X})(\|P_{\leq N}\u_0\|_{\X}+\|\pi_1 \stick_{t}^{N}(\xi)\|_{L^6})^{6} \\
& \hphantom{X} +C(1+\|P_{\leq N}\u_0\|_{\X}^{13}+\| \stick_{t}^{N}(\xi)\|_{\W^{\al,\frac{4}{\al}}}^{\frac{4}{\al}}+\| \wick{(\stick_{t}^{N}(\xi))^{2}}_{\g}\|_{L^4}^{7}+\| \wick{(\stick_{t}^{N}(\xi))^{3}}_{\g}\|_{L^2}^{2}) \\
& \leq -\frac{1}{5}F(\w_N) +2Ce^{-\frac{t}{8}}\| \u_0\|_{\X}F(\w_N) +C(1+\|P_{\leq N}\u_0\|_{\X}^{13}+\| \stick_{t}^{N}(\xi)\|_{\W^{\al,\frac{4}{\al}}}^{\frac{4}{\al}} \\
& \hphantom{X}+\| \wick{(\stick_{t}^{N}(\xi))^{2}}_{\g}\|_{L^4}^{7}+\| \wick{(\stick_{t}^{N}(\xi))^{3}}_{\g}\|_{L^2}^{2}).
\end{align*}
For 
\begin{align*}
t\le \tau(\|\u_0\|_{\X}):= 8\log \big( 20C\|\u_0\|_{\X}\vee 10   \big),
\end{align*}
Gronwall's inequality implies 
\begin{align*}
&\phantom{\le\ }F(\w_N (t))\\
&\leq e^{-2ct}e^{16C\|\u_0\|_{\X}}F(\w_{N}(0))+Ce^{32C\|\u_0\|_{\X}}(1+\|P_{\leq N}\u_0\|_{\X}^{13}) \\
& \hphantom{X}+ e^{32C\|\u_0\|_{\X}} \int_0^{t} e^{-2c(t-t')}(\| \stick_{t'}^{N}(\xi)\|_{\W^{\al,\frac{4}{\al}}}^{\frac{4}{\al}}+\| \wick{(\stick_{t'}^{N}(\xi))^{2}}_{\g}\|_{L^4}^{7}+\| \wick{(\stick_{t'}^{N}(\xi))^{3}}_{\g}\|_{L^2}^{2}) dt'.
\end{align*}
Hence, by \eqref{FbyE} and \eqref{EbyF}, we have 
\begin{align*}
E(\w_N(t))& \les e^{32C\|\u_0\|_{\X}}\bigg(    1+\|P_{\leq N}\u_0\|_{\X}^{13}+\|\pi_1 \stick_{t}^{N}(\xi)\|_{L^6}^{6}  \\
& \hphantom{X} +\int_0^{t} e^{-2c(t-t')}(\| \stick_{t'}^{N}(\xi)\|_{\W^{\al,\frac{4}{\al}}}^{\frac{4}{\al}}+\| \wick{(\stick_{t'}^{N}(\xi))^{2}}_{\g}\|_{L^4}^{7}+\| \wick{(\stick_{t'}^{N}(\xi))^{3}}_{\g}\|_{L^2}^{2}) dt'\bigg).
\end{align*}
For $t \ge \tau$,
we have 
\begin{align*}
\dt F(\w_N)  \leq -\frac{1}{10}F(\w_N)+C\big(&1+\|P_{\leq N}\u_0\|_{\X}^{13}+\| \stick_{t}^{N}(\xi)\|_{\W^{\al,\frac{4}{\al}}}^\frac 4 \al\\
&+\| \wick{(\stick_{t}^{N}(\xi))^{2}}_{\g}\|_{L^4}^{7}+\| \wick{(\stick_{t}^{N}(\xi))^{3}}_{\g}\|_{L^2}^{2}\big).
\end{align*}
Applying Gronwall's inequality with $c:=\frac{1}{10}$, we obtain
\begin{align*}
F(\w_N (t+\tau))& \leq e^{-ct}F(\w_N(\tau))+C(1+\|P_{\leq N}\u_0\|_{\X}^{13}) \\
&\hphantom{X} +\int_0^{t} e^{-c(t-t')}(\| \stick_{t'}^{N}(\xi)\|_{\W^{\al,\frac{4}{\al}}}^{\frac{4}{\al}} +\| \wick{(\stick_{t'}^{N}(\xi))^{2}}_{\g}\|_{L^4}^{7}+\| \wick{(\stick_{t'}^{N}(\xi))^{3}}_{\g}\|_{L^2}^{2}) dt'.
\end{align*}
Therefore, by \eqref{FbyE} and \eqref{EbyF},
\begin{align*}
E&(\w_N(t)) \les F(\w_N(t))+(\|P_{\leq N}\u_0\|_{\X}+\|\pi_1 \stick_{t}^{N}(\xi)\|_{L^6})^{6} \\
& \les e^{-ct}F(\w_N(\tau))+ \int_0^{t} e^{-c(t-t')}(\| \stick_{t'}^{N}(\xi)\|_{\W^{\al,\frac{4}{\al}}}^{\frac{4}{\al}}+\| \wick{(\stick_{t'}^{N}(\xi))^{2}}_{\g}\|_{L^4}^{7}+\| \wick{(\stick_{t'}^{N}(\xi))^{3}}_{\g}\|_{L^2}^{2}) dt' \\
& \hphantom{X} +1+\|P_{\leq N}\u_0\|_{\X}^{13}+\|P_{\leq N}\u_0\|_{\X}^{6}+\|\pi_1 \stick_{t}^{N}(\xi)\|_{L^6}^{6} \\
& \les  (e^{-ct} e^{32C\|\u_0\|_{\X}}+ 1)\\
& \phantom{\les\ }\times \Big(\int_0^{t} e^{-c(t-t')}(\| \stick_{t'}^{N}(\xi)\|_{\W^{\al,\frac 4 \al}}^{\frac{4}{\al}}+\| \wick{(\stick_{t'}^{N}(\xi))^{2}}_{\g}\|_{L^4}^{7}+\| \wick{(\stick_{t'}^{N}(\xi))^{3}}_{\g}\|_{L^2}^{2}) dt'\Big) \\
& \hphantom{X} +1+\|P_{\leq N}\u_0\|_{\X}^{13}+\|P_{\leq N}\u_0\|_{\X}^{6}+\|\pi_1 \stick_{t}^{N}(\xi)\|_{L^6}^{6} \\
& \les \eta(t;\|\u_0\|_{\X}) \Big(1+\|P_{\leq N}\u_0\|_{\X}^{13}+\|\pi_1 \stick_{t}^{N}(\xi)\|_{L^{6}}^{6} \\
& \hphantom{X}+ \int_0^{t} e^{-c(t-t')}(\| \stick_{t'}^{N}(\xi)\|_{\W^{\al,\frac{4}{\al}}}^{\frac{4}{\al}}+\| \wick{(\stick_{t'}^{N}(\xi))^{2}}_{\g}\|_{L^4}^{7}+\| \wick{(\stick_{t'}^{N}(\xi))^{3}}_{\g}\|_{L^2}^{2}) dt'\Big).
\end{align*}
\end{proof}

\section{Semigroup property of the flow and existence of an invariant measure}

Given $\u_0\in \X$, we write $\Phi_{t}(\u_0; \xi)=S(t)\u_0+\stick_{t}(\xi)+\v$, where $\v$ solves \eqref{veqn}.
We denote by $\B_{b}(\X)$ the space of measurable, bounded functions from $\X$ to $\R$, and by $\mathcal{C}_{b}(\X)$ the space of bounded continuous functions from $\X$ to $\R$, both endowed with the norm 
\begin{align*}
\|F\|_{L^{\infty}}:=\sup_{\u_0 \in \X}|F(\u_0)|.
\end{align*}
 Given $F\in \mathcal{B}_{b}(\X)$, we define the family of bounded linear operators $\Pt{t}:F\mapsto \Pt{t}F$ for $t\geq 0$, by 
\begin{align}
\Pt{t}F(\u_0):=\E[ F(\Phi_{t}(\u_0;\xi))], \label{Ptdefn}
\end{align}
for every $\u_0\in \X$. 
We will show that $\{ \Phi_{t} \}_{t\geq 0}$ is a Markov process and $\{\Pt{t}\}_{t\geq 0}$ is a Markov semigroup with respect to the filtration $\{\F_{t}\}_{t\geq 0}$. 

\begin{proposition}\label{PROP:markov}
Let $\u_0\in \X$ and $\Phi_{t}(\u_0;\xi)$ be as given above. Then for every $F\in \B_{b}(\X)$ and $t\geq 0$, we have 
\begin{align}
\E\big[ F(\Phi_{t+h}(\u_0;\xi))\vert \F_{t}\big] = \Pt{h}F(\Phi_{t}(\u_0;\xi)) \label{markov}
\end{align}
for all $h\geq 0$. In particular, $\{ \Phi_{t} \}_{t\geq 0}$ is a Markov process and $\{\Pt{t}\}_{t\geq 0}$ is a Markov semigroup.
\end{proposition}
\begin{proof}
We use the same argument as in~\cite[Section 4.1]{TW}. We first introduce some notation.
Given $h>0$, we define 
\begin{align}
\stick_{t,t+h}(\xi):=\int_t^{t+h} S(t+h-t')\vec{0}{\sqrt{2}\jb{\nb}^{-s}\xi(t')} dt', \label{sticksplit}
\end{align}
which is the solution to the linear damped wave equation \eqref{NLWvec} started from time $t$ with zero initial data, and similarly, 
\begin{align}
\begin{split}
\v_{t,t+h}&:= \v_{t,t+h}(\Phi_{t}(\u_0;\xi), \stick_{t,t+h}(\xi)) \\
& =-\int_t^{t+h} S(t+h-t') \vec{0}{\NN_{\g}[S(t'-t)\Phi_{t}(\u_0)+\stick_{t,t+t'}(\xi)+\v_{t,t+t'})]  } dt',
\label{vsplit}
\end{split}
\end{align}
which is the solution to \eqref{veqn} starting at time $t$, driven by $S(\cdot)\Phi_{t}(\u_0)$ and $\stick_{t,t+\cdot}$. 
We first claim that
\begin{align}
\Phi_{t+h}(\u_0,\xi)=S(h)\Phi_{t}(\u_0,\xi)+\stick_{t,t+h}(\xi)+\v_{t,t+h}(\Phi_{t}(\u_0,\xi), \stick_{t,t+h}(\xi)). \label{flowsplit}
\end{align}
From \eqref{sticksplit}, we have
\begin{align}
\stick_{0,t+h}(\xi)=\stick_{t,t+h}(\xi)+S(h)\stick_{0,t}(\xi), \label{sticksplit2}
\end{align}
while from \eqref{vsplit}, we have
\begin{align}
\v_{0,t+h}(\u_0,\stick_{0,t+h})&=S(h)\v_{0,t}(\u_0,\stick_{0,t})\notag\\
&\phantom{=\ }+\int_{t}^{t+h}S(t+h-t') \vec{0}{\NN_{\g}[S(t')\u_0+\stick_{0,t'}(\xi)+\v_{0,t'})]  }dt'. \label{vsplit2}
\end{align}
Using \eqref{sticksplit2}, we have 
\begin{align*}
S(t')\u_0+\stick_{0,t'}(\xi)+\v_{0,t'}=S(t'-t)\Phi_{t}(\u_0,\xi)+\stick_{t,t'}(\xi)+\v_{0,t'}-S(t'-t)\v_{0,t},
\end{align*}
so by \eqref{vsplit2},
\begin{align*}
&\phantom{=\ }\v_{0,t+h}(\u_0,\stick_{0,t+h})- S(h)\v_{0,t}(\u_0,\stick_{0,t})\\
&= \int_{t}^{t+h}S(t+h-t') \vec{0}{\NN_{\g}[S(t'-t)\Phi_{t}(\u_0,\xi)+\stick_{t,t'}(\xi)+\v_{0,t'}-S(t'-t)\v_{0,t}]  }dt'.
\end{align*}
This shows that $\v_{0,t+h}(\u_0,\stick_{0,t+h})- S(h)\v_{0,t}(\u_0,\stick_{0,t})$ solves the equation \eqref{vsplit}. Therefore, by uniqueness of solutions (see Proposition \ref{PROP:LWP}), we have that 
\begin{align*}
\v_{0,t+h}(\u_0,\stick_{0,t+h})=S(h)\v_{0,t}(\u_0,\stick_{0,t})+\v_{t,t+h}(\Phi_{t}(\u_0),\stick_{t,t+h}),
\end{align*}
so \eqref{flowsplit} holds.
Now, \eqref{flowsplit} allows us to view $\Phi_{t+h}(\u_0,\xi)$ as a function of $\stick_{t,t+h}$ and $\Phi_{t}(\u_0,\xi)$, where the former is independent from $\F_{t}$, and the latter is $\F_{t}$-measurable. Therefore, by \cite[Proposition 1.12]{DPZ1}, we have 
\begin{align} \label{condExp}
\E\big[ F(\Phi_{t+h}(\u_0;\xi))\vert \F_{t}\big] = \Psi(\Phi_{t}(\u_0,\xi)),
\end{align}
where 
\begin{align*}
\Psi(\u_0) = \E\big[ F(S(h)\u_0+ \stick_{t,t+h} + \v_{t,t+h}(\u_0, \stick_{t,t+h}) \big].
\end{align*}
By performing the change of variable $ t' = t + s'$ in \eqref{vsplit}, we see that $\v_{t_0,t}(\u_0, \stick_{t_0,t_0 + t})$ solves \eqref{veqnd} with $\stick_{t}$ replaced by $\stick_{t_0, t_0 + t}$ and $\v_0 = 0$. Therefore, by uniqueness of solutions to \eqref{veqnd} (see Proposition \ref{PROP:LWP}), we have that $\v_{t_0,t}(\u_0, \stick_{t_0,t_0 + t}) = \v_{0,t}(\u_0, \stick_{t_0,t_0 + t})$. Recalling that $\mathcal L( \stick_{t_0,t_0 + t}) = \mathcal L(\stick_{t})$, we have that 
\begin{align*}
\Psi(\u_0)&= \E\big[ F(S(h)\u_0+ \stick_{t,t+h} + \v_{t,t+h}(\u_0, \stick_{t,t+h}) \big] \\
&= \E\big[ F(S(h)\u_0+ \stick_{h} + \v_{0,h}(\u_0, \stick_{0,h}) \big]\\
&=\E\big[ F(\Phi_h(\u_0, \xi))\big]\\
&=\Pt{h}(\u_0).
\end{align*}
By plugging this into \eqref{condExp}, we obtain \eqref{markov}.
Finally, the semigroup property for $\Pt{t}$ follows from \eqref{markov} since for any $\u_0\in \X$ and $F\in \B_{b}(\X)$, we have
\begin{align*}
\Pt{t+h}F(\u_0)& = \E \big[F\big(\Phi_{t+h}(\u_0;\xi)\big) \big] \\
& = \E \big[\E\big[ F(\Phi_{t+h}(\u_0;\xi))\vert \F_{t}\big]  \big]  \\
&  =\E \big[ \Pt{h}F(\Phi_{t}(\u_0;\xi))\big] \\
& = \Pt{t}\Pt{h} F(\u_0).
\end{align*}
\end{proof}

\begin{proposition}\label{PROP:Feller}
The Markov semigroup $\{\Pt{t}\}_{t\geq 0}$ in \eqref{Ptdefn} is Feller on $\X$. More precisely, if $F\in \CC_{b}(\X)$, then $\Pt{t}F\in \CC_{b}(\X)$ for any $t\geq 0$.
\end{proposition}

\begin{proof}
It follows from Lemma~\ref{LEM:ctyofflow} and \eqref{S(t)onX} that the map $\u_0 \mapsto \Phi_{t}(\u_0,\xi)$ is continuous. The claim then follows from the dominated convergence theorem. 
\end{proof}

\begin{proposition} \label{PROP:KryBo}
For every $\u_0 \in \X$, there exists a probability measure $\mu$ on $\X$ and a sequence $t_{k}\to \infty$ such that 
\begin{align*}
\frac{1}{t_{k}}\int_{0}^{t_{k}}\Pt{t}^{\ast}\delta_{\u_0} \,  dt \wto \mu.
\end{align*}
In particular, $\mu$ is invariant for the semigroup $\{\Pt{t}\}_{t\geq 0}$ on $\X$.
\end{proposition}

\begin{proof}
 We write $\Phi_{t}( {\u_0},\xi)=S(t) \u_0 + \stick_{t}(\xi)+\v(t)$, where $\v$ solves \eqref{veqn} and by \eqref{SonH} and \eqref{veqn} satisfies
\begin{align}
\|\v(t)\|_{\H^{1+\al}} \leq C\int_{0}^{t}e^{-\frac{t-t'}{2}}\big(\|S(t') \u_0\|_{\X}^3 + \|\stick_{t'}(\xi)\|_{Z^{\al}}^{3}+\|\v(t')\|_{\H^{1}}^{3} \big) dt'. \label{tightv}
\end{align}
Let $\s :=\frac{\al+s}{2}\in (\al,s)$.
Given $R>0$, let $B_{R}$ denote the closed ball of radius $R$, centred at zero in $Z^{\s}+\mathcal{H}^{1+\al}\subset \X$. From Lemma~\ref{LEM:Hcompact} and Lemma~\ref{LEM:Zcompact}, $B_{R}\subset Z^{\s}+\H^{1+\al}$ is compact in $\X$. Moreover, by Lemma \ref{compactOrbits}, the set $O_{\u_0} = \{ S(t)\u_0: t \ge 0\} \cup \{{\bf 0}\}$ defined in \eqref{orbit} is also compact in $\X$. 
Then, 
\begin{align*}
\Pt{t}^{\ast}\dl_{\u_0}\big((O_{\u_0} + B_{R})^{c}\big)&=\L( \Phi_{t}( {\u_0},\xi))\big((O_{\u_0} + B_{R})^{c}\big) \\
& \le \prob\big( \|\stick_{t}(\xi)\|_{Z^{\s}}+\|\v(t)\|_{\H^{1+\al}}>R\big) \\
& \leq 2R^{-1}\,\E\big[  \|\stick_{t}(\xi)\|_{Z^{\s}}\big]+2R^{-1}\,\E\big[  \|\v(t)\|_{\H^{1+\al}}\big].
\end{align*}
The bound \eqref{vmoments} from Proposition~\ref{PROP:GWP} and \eqref{momentsstick} imply $\E[\|\v(t)\|^3_{\H^{1}}] \leq C = C(\u_0)$, uniformly in $t\geq0$. Thus, by \eqref{momentsstick} and \eqref{tightv}, we have 
\begin{align*}
\Pt{t}^{\ast}\dl_{\u_0}\big((O_{\u_0} + B_{R})^{c}\big) \leq \frac{C'}{R}.
\end{align*}
For every $\ep>0$, there exists $R_{\ep}$ such that 
\begin{align*}
\Pt{t}^{\ast}\dl_{\u_0}\big((O_{\u_0} + B_{R_\eps})^{c}\big) < \ep
\end{align*}
for any $t> 0$. Therefore, $\{ \Pt{t}^{\ast}\dl_{\u_0}\}_{t\ge 0}$ is tight in $\X$ and so to is the family of probability measures 
\begin{align*}
\mu_{t}:=\frac{1}{t}\int_{0}^{t}\Pt{t'}^{\ast}\dl_{\u_0}\, dt'.
\end{align*}
Prokhorov's theorem implies there exists a probability measure $\mu$ and a sequence $t_{k}\to \infty$ such that $\mu_{t_k}$ converges weakly to $\mu$ as $k\to \infty$.
Given $\phi\in \CC_{b}(\X)$, Proposition~\ref{PROP:Feller} implies $\Pt{t}\phi\in \CC_{b}(\X)$ for any $t>0$. 
Thus for any $t>0$, we have 
\begin{align*}
\jb{\Pt{t}^{\ast}\mu,\phi}&= \lim_{k\to \infty} \frac{1}{t_k} \int_{0}^{t_k}\jb{\Pt{t}^\ast\Pt{s}^{\ast}\dl_{\u_0},\phi}ds \\
& =  \lim_{k\to \infty} \frac{1}{t_k} \int_{t}^{t+t_k}\jb{\Pt{s}^{\ast}\dl_{\u_0},\phi}ds \\
&= \lim_{k\to \infty}\bigg[ \frac{1}{t_k}\int_{0}^{t_{k}}\jb{\Pt{s}^{\ast}\dl_{\u_0},\phi}ds + \frac{1}{t_k}\int_{t_k}^{t+t_k}\jb{\Pt{s}^{\ast}\dl_{\u_0},\phi}ds-\frac{1}{t_k}\int_{0}^{t}\jb{\Pt{s}^{\ast}\dl_{\u_0},\phi}ds\bigg]\\
& = \jb{\mu,\phi},
\end{align*} 
since the last two terms in the third equality above can be bounded by $\frac{2t}{t_k}\|\phi\|_{L^\infty}\to 0$ as $k\to \infty$. This shows that $\mu$ is invariant for $\{\Pt{t}\}_{t>0}$.
\end{proof}

\section{Uniqueness of the invariant measure} 
As discussed in the introduction, the semigroup $\Pt{t}$ defined in \eqref{Ptdefn} is not strong Feller. Indeed, the following holds.

\begin{proposition}\label{PROP:notsf}
Let $\alpha < \alpha' < s\wedge 1$. Let $F(\u_0):= \ind_{\{\u_0 \in \H^{\alpha'}\}}$. For every $t \ge 0$, we have that 
\begin{equation*}
\E[F(\Phi_t(\u_0,\xi))] = F(\u_0).
\end{equation*}
In particular, $\Pt{t} F(\u_0)$ is an everywhere discontinuous function of $\u_0 \in \X$.
\end{proposition}
\begin{proof}
By Proposition \ref{PROP:stickZ}, we have that for every $t \ge 0$, $\stick_t(\xi) \in Z^{\alpha'} \subset \H^{\alpha'}$. Therefore, by the decomposition \eqref{soldef} and the fact that $\v(t) \in \H^1\subset \H^{\alpha'}$ for every $t \ge 0$, we have that a.s.
\begin{equation*}
\Phi_t(\u_0,\xi) \in \H^{\alpha'} \iff S(t)\u_0 \in \H^{\alpha'}.
\end{equation*}
Since the operator $S(t)$ is invertible on $\H^{\s}$ for every $\s\in\R$, this implies that 
\begin{equation*}
\Phi_t(\u_0,\xi) \in \H^{\alpha'} \iff S(t)\u_0 \in \H^{\alpha'} \iff \u_0 \in \H^{\alpha'} \quad \text{a.s.},
\end{equation*}
or equivalently, 
\begin{equation*}
\Pt{t}F(\u_0) = \E[F(\Phi_t(\u_0,\xi))] = F(\u_0).
\end{equation*}
As discussed in Remark \ref{Xnecessity}, $\X \not \subseteq \H^{\alpha'}$, and $\H^{\alpha'} \cap \X$ contains the trigonometric polynomials. Therefore, $\H^{\alpha'} \cap \X$ is a nontrivial dense subspace of $\X$, so $F$ is discontinuous everywhere.
\end{proof}

Therefore, in order to show uniqueness of the invariant measure, we follow the strategy introduced by Hairer and Mattingly in \cite{HM1}, and further developed by the same authors and Scheutzow in \cite{HM3}. We also refer to \cite{BW} for a general approach in the abstract setting. 
While it is possible for the proof of Theorem \ref{THM:erg} to be included in the framework of \cite{BW}, we believe that pursuing a direct proof is both simpler and would be better for the exposition. For the same reason, we will try to keep the proof as self-contained as possible.

\subsection{Abstract definitions}
We recall the following definitions from \cite[Section 3.1]{HM1}. 
Let $\mathcal{X}$ be a Polish space (i.e. complete, separable and metrizable). A pseudo-metric on $\mathcal{X}$ is a map $d: \mathcal{X}\times \mathcal{X} \mapsto \R$ such that $d(x,x)=0$ and $d$ satisfies the triangle inequality. For a sequence $\{d_n\}_{n \in \N}$ of (pseudo)-metrics on $\mathcal{X}$, we say that $\{d_n\}_{n \in \N}$ is totally separating if for every $x,y \in \mathcal{X}$ with $x \neq y$,
\begin{equation*}
\lim_{n \to \infty} d_n(x,y) = 1.
\end{equation*}
Given a pseudo-metric $d$, we define the following seminorm on the set of $d-$Lipschitz continuous functions from a Polish space $\mathcal{X}$ to $\R$: 
\begin{align}
\| F\|_{d}=\sup_{ \substack{x,y\in \mathcal{X} \\ x\neq y}} \frac{ |F(x)-F(y)|}{d(x,y)}.  
\label{lipd}
\end{align} 
This defines a dual seminorm on the space of finite signed Borel measures on $\mathcal{X}$ with vanishing integral by 
\begin{align}
\| \nu \|_{d}=\sup_{\| F\|_{d}\leq 1} \int_{\mathcal{X}}F(x)d\nu(x). \label{measd}
\end{align}
In the following, we will consider $\{d_{n}\}_{n\in \N}$ to be the following totally separating system of pseudo-metrics on $\mathcal{X}$:
\begin{align}
d_{n}(x,y)=1\wedge n \|x-y\|_{\mathcal{X}}. \label{dn}
\end{align}
Notice that for any $x,y\in \mathcal{X}$ and $n\in \N$, $d_{n}(x,y)\leq 1$.

When we work with the metrics $d_n$, we can further restrict the space of test functions $F$ in the definition of the norm in \eqref{measd}.
\begin{lemma}\label{LEM:testLinfty} Fix $n\in \N$ and let $\nu$ be a finite signed Borel measure on $\mathcal{X}$ with vanishing integral. Then 
\begin{align}
\| \nu \|_{d_{n}}=\sup_{ \substack{ \|F\|_{d_{n}}\leq 1 \\ \|F\|_{L^{\infty}}\leq \frac{1}{2}}}\int_{\mathcal{X}} F(x)d\nu(x), \label{measdn12}
\end{align}
where the norm on the left hand side of \eqref{measdn12} is defined in \eqref{measd}.
\end{lemma}

\begin{proof}
We first show that 
\begin{align}
\sup_{\| F\|_{d_n}\leq 1} \int_{\mathcal{X}}F(x)d\nu(x) = \sup_{ \substack{ \|F\|_{d_{n}}\leq 1 \\ \|F\|_{L^{\infty}}\leq 1}}\int_{\mathcal{X}} F(x)d\nu(x). \label{sup1}
\end{align}
It suffices to show 
\begin{align}
\sup_{\| F\|_{d_n}\leq 1} \int_{\mathcal{X}}F(x)d\nu(x) \leq  \sup_{ \substack{ \|F\|_{d_{n}}\leq 1 \\ \|F\|_{L^{\infty}}\leq 1}} \int_{\mathcal{X}} F(x)d\nu(x). \label{sup11}
\end{align}
Let $F$ be such that $\|F\|_{d_n} \leq 1$. Given any $x_0 \in \mathcal{X}$, set $\wt{F}(x):=F(x)-F(x_0)$. Then, for $x\neq x_0$, 
\begin{align*}
|\wt{F}(x)| =|F(x)-F(x_0)| \leq \|F\|_{d_{n}}d_{n}(x,x_0) \leq 1,
\end{align*}
since $d_{n}(x,x_0)\leq 1$. Therefore, $\| \wt{F}\|_{L^{\infty}}\leq 1$. Since $\nu$ has vanishing integral, 
\begin{align*}
\int_{\mathcal{X}}F(x)d\nu(x) = \int_{\mathcal{X}} \wt{F}(x)d\nu(x) \leq \sup_{ \substack{ \|\wt{F}\|_{d_{n}}\leq 1 \\ \|\wt{F}\|_{L^{\infty}}\leq 1}} \int_{\mathcal{X}} \wt{F}(x)d\nu(x).
\end{align*} 
As this is true for any such $F$, this proves \eqref{sup11} and so also \eqref{sup1}. 
In order to show \eqref{measdn12}, it suffices to show 
\begin{align}
\sup_{ \substack{ \|F\|_{d_{n}}\leq 1 \\ \|F\|_{L^{\infty}}\leq 1}}\int_{\mathcal{X}} F(x)d\nu(x)\leq \sup_{ \substack{ \|F\|_{d_{n}}\leq 1 \\ \|F\|_{L^{\infty}}\leq \frac{1}{2}}}\int_{\mathcal{X}} F(x)d\nu(x). \label{sup12}
\end{align}
To this end, let $F$ be such that $\|F\|_{d_{n}}\leq 1$ and $\|F\|_{L^{\infty}}\leq 1$.
We have that 
\begin{align*}
\sup F - \inf F = \sup_{x,y \in \mathcal X} F(x) - F(y) \le \sup_{x,y \in \mathcal X} \|F\|_{d_n} d_n(x,y) \le 1.
\end{align*}

Now, $\wt{F}(x):=F(x)-\frac{\sup F + \inf F}{2}$ satisfies $\|\wt{F}\|_{d_n}\leq 1$ and $\| \wt{F}\|_{L^{\infty}}\leq \frac{1}{2}$. We show the latter inequality. Indeed, since $\sup F -\inf F\leq 1$, we have 
\begin{align*}
\wt{F}(x) \geq \inf F - \frac{\sup F +\inf F}{2}= -\frac{( \sup F -\inf F)}{2}\geq -\frac{1}{2}
\end{align*}
and 
\begin{align*}
\wt{F}(x)\leq \sup F -\frac{ \sup F +\inf F}{2} =\frac{\sup F -\inf F}{2} \leq \frac{1}{2}.
\end{align*}
Finally, \eqref{sup12} follows from $\int_{\mathcal{X}} Fd\nu = \int_{\mathcal{X}} \wt{F} d\nu$.
\end{proof}

Given $\mu_1$ and $\mu_2$, two positive finite Borel measures on $\mathcal{X}$ of equal mass, we denote by $\Pi(\mu_1,\mu_2)$ the set of positive measures on $\mathcal{X}^2$ with marginals $\mu_1$ and $\mu_2$. We define
\begin{align*}
\vvv\mu_1-\mu_2\vvv_{d} = \inf_{\pi \in \Pi(\mu_1,\mu_2)} \int_{\mathcal{X}^2} d(x,y) d\pi(x,y).
\end{align*}
\begin{lemma}\label{LEM:MKDuality}
Let $d$ be a continuous pseudo-metric on a Polish space $\mathcal{X}$ and let $\mu_1$ and $\mu_2$ be two positive measures on $\mathcal{X}$ with equal mass. Then, $\| \mu_1-\mu_2\|_{d}=\vvv \mu_1-\mu_2\vvv_{d}.$
\end{lemma}
\noi
See \cite[Lemma 3.3]{HM1} for a proof which is based on the Monge-Kantorovich duality. 

\begin{lemma}\label{LEM:dnlimit}
Let $\mathcal{X}$ be a Polish space and let $\{d_n\}_{n\in \N}$ be a totally separating system of pseudo-metrics for $\mathcal{X}$. Then,
\begin{align*}
\|\mu_1-\mu_2\|_{\textup{TV}}=\lim_{n\to \infty} \|\mu_1-\mu_2\|_{d_n},
\end{align*}
for any two positive measures $\mu_1$ and $\mu_2$ with equal mass on $\mathcal{X}$.
\end{lemma}
\noi
Similarly, see \cite[Corollary 3.5]{HM1} for the proof of this statement.

\subsection{Construction of the Girsanov shift}

In the following, we define $p_t$ to be the heat kernel at time $t$ on the torus $\T^2$, i.e.\ 
\begin{align*}
(e^{t \Delta} f) (x) = \int_{\T^2} p_t(x-y) f(y) d y.
\end{align*}
We recall the following regularisation effect of the heat kernel.
\begin{lemma}\label{LEM:convdiff}
Let $1\le p\le \infty$ and $0 \le \ta \le 2$. We have 
\begin{align*}
\| f-f\ast \rho_{\ep}\|_{L^{p}(\T^2)} \leq C\ep^{\frac \ta 2}\|f\|_{W^{\ta ,p}(\T^2)}, 
\end{align*}
where $C = C(\ta) >0$ is independent of $\ep$. Moreover, for $\ta \ge 0$, we have 
\begin{equation} \label{smoothing}
\| f \ast \rho_\eps\|_{L^p (\T^2)} \les \ep^{-\frac \ta 2} \| f \|_{W^{-\ta, p} (\T^2)}.
\end{equation}
\end{lemma}

Our main goal in this Subsection is to prove Propostion~\ref{PROP:shift} and Proposition~\ref{LEM:h}, below. To this end, we first prove an auxiliary estimate. We define 
\begin{align}
\Q(\u,\v):= 3((\pi_1 \u)^2-\g)+3\pi_1 \u\pi_1 \v+(\pi_1 \v)^2, \label{Qdefn}
\end{align}
so that 
\begin{align*}
\NN_{\g}(\u+\v)-\NN_{\g}(\u)=\Q(\u,\v)\pi_1 \v.
\end{align*}

\begin{lemma}\label{LEM:epswelldef} Let $s>0$, $\g\in \R$, $\u_1^{0}, \u_2^{0}\in \X$ and $\u_1$ be the global solution to \eqref{duhamel} with $\u_1 \vert_{t=0}=\u_1^{0}$ and 
of the form 
\begin{align}
\u_1 (t) = S(t) \u_1^{0}+\stick_{t}(\xi)+\v_1(t; \u_1^{0}, \xi). \label{u1expQ}
\end{align}
Let $\mathbf{z} \in L^p(\O; \H^1)$.
Then, there exists $C=C(\|\u_1^{0}\|_{\X}, \|\u_2^{0}\|_{\X})>0$ such that for any $1\leq p<\infty$, we have
\begin{align}
\E\big[ \| \Q( \u_1(t), \mathbf{z}+S(t)\u_2^{0})\|_{W^{\al,\frac{2}{1-\al}}}^{p}   \big]^{\frac{1}{p}}\leq C+\E\big[ \| \mathbf{z}\|_{\H^1}^{p} \big]^{\frac{2}{p}}, \label{Qbound}
\end{align}
uniformly in $t \ge 0$.
\end{lemma}
\begin{proof} 
We begin by proving some basic, but useful, estimates. Given $\fb \in \X$ and $\gb, \gb_1, \gb_2\in \H^1$, we have
\begin{align} 
\| (\pi_1 S(t) \fb)^2 \|_{W^{\al, \frac{2}{1-\al}}} & \les e^{-\frac{t}{4}} \|\fb\|_{\X}^{2}, \label{basic1}\\
\| (\pi_1 S(t)\fb) \pi_1 \gb\|_{W^{\al, \frac{2}{1-\al}}}& \les e^{-\frac{t}{8}}\|\fb\|_{\X}\| \gb\|_{\H^1},\label{basic2} \\
\| (\pi_1 S(t)\fb)\pi_1 \stick_{t}(\xi)\|_{W^{\al,\frac{2}{1-\al}}}& \les e^{-\frac{t}{8}}\|\fb\|_{\X}\| \stick_{t}(\xi)\|_{\W^{\al, \frac{2}{1-2\al}}}, \label{basic3}\\
\| \pi_1 \gb_1 \pi_1 \gb_2\|_{W^{\al,\frac{2}{1-\al}}}& \les \|\gb_1\|_{\H^1}\|\gb_2\|_{\H^1}.\label{basic4}
\end{align}
for any $t>0$.
Using \eqref{fraclieb}, H\"{o}lder's inequality (since $\al<\frac13$) and \eqref{Xalpha}, we have
\begin{align*}
\| (\pi_1 S(t) \fb)^2 \|_{W^{\al, \frac{2}{1-\al}}} \les \| \jb{\nb}^{\al}\pi_1 S(t) \fb\|_{L^{\frac{2}{\al}}}\|\pi_1 S(t)\fb\|_{L^{\frac{2}{1-2\al}}} \les e^{-\frac{t}{4}}\|\fb\|_{\X}^{2}.
\end{align*}
This shows \eqref{basic1}. To show \eqref{basic2} and \eqref{basic4}, we proceed in the same way, recalling the embedding $\H^1 \embeds \X$ (see Lemma \ref{LEM:Hcompact}).
 Finally, for \eqref{basic3}, we use \eqref{fraclieb} and Sobolev embedding to obtain
\begin{align*}
&\phantom{\les\ }\| (\pi_1 S(t)\fb)\pi_1 \stick_{t}(\xi)\|_{W^{\al,\frac{2}{1-\al}}} \\
& \les \| \jb{\nb}^{\al}\pi_1 S(t) \fb\|_{L^{\frac{2}{\al}}}\|\pi_1 \stick_t(\xi)\|_{L^{\frac{2}{1-2\al}}} + \| \pi_1 S(t) \fb\|_{L^{\frac{2}{\al}}}\|\jb{\nb}^{\al} \pi_1 \stick_t(\xi)\|_{L^{\frac{2}{1-2\al}}} \\
&\les \| \pi_1 S(t) \fb\|_{W^{\al, \frac{2}{\al}}}\|\pi_1 \stick_t(\xi)\|_{W^{\al, \frac{2}{1-2\al}}} \\
& \les e^{-\frac{t}{8}}\|\fb\|_{\X}\| \stick_{t}(\xi)\|_{\W^{\al, \frac{2}{1-2\al}}}.
\end{align*}
It follows from \eqref{Qdefn}, \eqref{u1expQ}, \eqref{basic1}, \eqref{basic2}, \eqref{basic3}, \eqref{basic4}, 
\begin{align*}
 \| \Q( \u_1(t), \mathbf{z}+S(t)\u_2^{0})\|_{W^{\al,\frac{2}{1-\al}}} & \les \|\mathbf{z}\|_{\H^1}^2+\|\v_1(t)\|_{\H^1}^{2}+e^{-\frac{t}{4}}\big( \|\u_1^{0}\|_{\X}+\|\u_2^{0}\|_{\X}\big) \\
 &\hphantom{XXXXXXX}+ \|\stick_{t}(\xi)\|_{\W^{\al,6}}^{2}+\|\wick{(\stick_{t}(\xi))^{2}}_{\g}\|_{\W^{\al,\frac{2}{1-\al}}}.
\end{align*}
Now, \eqref{Qbound} follows from the above estimate combined with \eqref{momentsstick} and \eqref{vmoments}.
\end{proof}

\begin{proposition}\label{PROP:shift}
Let $s>0$, $\g \in \R$ and $\u_1^0, \u_2^0\in \X$, and $\u_1$ be the global solution to \eqref{duhamel} with $\u_1 \vert_{t=0}=\u_1^0$ 
of the form 
\begin{align}
\u_1 (t) = S(t) \u_1^0+\stick_{t}(\xi)+\v_1(t; \u_1^{0}, \xi). \label{u1inw}
\end{align} Let $\w^0\in \H^1$.
Then, the equation 
\begin{align}
\begin{split}
&\w(t)\\
 &=S(t)\w^{0}-\int_{0}^{t} S(t-t') \vec{0}{ \Q(\u_1, \w+S(t)(\u_2^0-\u_1^0)) (\pi_1 \w+\pi_1 S(t)(\u_2^{0}-\u_1^{0}))}dt' \\
 &+\int_{0}^{t} S(t-t') \vec{0}{\big(\Q(\u_1, \w+S(t)(\u_2^{0}-\u_1^{0})) \ast \rho_{\ep(\w)}\big)( \pi_1\w +\pi_1 S(t)(\u_2^{0}-\u_1^{0})\ast \rho_{\ep(\w)})}dt',
 \end{split}
 \label{weqn}
\end{align}
where, for $C\ge 1$ a universal constant, 
\begin{align}
\ep (\w)(t):= C^{-\frac 2 \al}(&1 + \|\w(t)\|_{\H^1} + \|(\u_2^{0}-\u_1^{0})\|_{\X}\nonumber \\
&+ \| \u_1(t)\|_{\X} + \| \Q(\u_1(t), \w(t)+S(t)(\u_2^{0}-\u_1^{0}))\|_{W^{\al,\frac{2}{1-\al}}})^{-\frac{4}{\al}},  
\label{epsilon}
\end{align}
is almost surely globally well-posed in $\H^{1}(\T^2)$ and, moreover, we have 
\begin{align}
\|\w(t)\|_{\H^1} \les e^{-\frac{1}{16}t}\|\u_2^{0}-\u_1^{0}\|_{\X}+e^{-\frac{1}{16}t}\|\w^{0}\|_{\H^1}.
\label{wbound}
\end{align}
\end{proposition}

\begin{proof}
For simplicity, we introduce the shorthands $\Q_{\w}$ for $\Q(\u_1, \w+S(t)(\u_2^{0}-\u_1^{0}))$ and $\wt \u^{0}:=\u_2^{0}-\u_1^{0}$. 
Let $Y \subset C([0,+\infty); \H^1)$ be the Banach space 
\begin{equation*}
Y = \{ \w \in  C([0,+\infty); \H^1): \|\w(t)\|_{\H^1} \les e^{-\frac t {16}} \}, \quad \| \w \|_{Y} := \sup_{t \ge 0} e^{\frac{t}{16}} \| \w\|_{\H^1}.
\end{equation*}
Define the operator $\Gamma: Y \to C([0,+\infty); \H^1)$ by 
\begin{align*}
&\Gamma[\w](t) \\
 &:=S(t)\w^{0}-\int_{0}^{t} S(t-t') \vec{0}{ \Q(\u_1, \w+S(t)(\u_2^0-\u_1^0)) (\pi_1 \w+\pi_1 S(t)(\u_2^{0}-\u_1^{0}))}dt' \\
 &+\int_{0}^{t} S(t-t') \vec{0}{\big(\Q(\u_1, \w+S(t)(\u_2^{0}-\u_1^{0})) \ast \rho_{\ep(\w)}\big)( \pi_1\w +\pi_1 S(t)(\u_2^{0}-\u_1^{0})\ast \rho_{\ep(\w)})}dt'.
\end{align*}
We want to show that if $C$ is a big enough constant, then $\Gamma$ is a contraction on $Y$ with $\Lip(\Gamma) \le \frac 12$.
We write 
\begin{align*}
F(\w)&:=\Q_{\w} (\pi_1 \w+\pi_1 S(t) \wt \u^{0})-\Q_{\w}\ast \rho_{\ep(\w)} (\pi_1 \w+\pi_1 S(t) \wt \u^{0}\ast \rho_{\ep(\w)}).
\end{align*}
We want to estimate $F(\w) - F(\w')$. We can assume without loss of generality that \linebreak $\ep(\w)(t)\le\ep(\w')(t)$. We have that 
\begin{align}
&F(\w) - F(\w') \nonumber \\
=&~[(\Q_\w - \Q_{\w'}) - (\Q_\w - \Q_{\w'}) \ast \rho_{\ep(\w)}](\pi_1 \w+\pi_1 S(t) \wt \u^{0}) \label{wdiff1}\\
&+ [\Q_{\w'} \ast \rho_{\ep(\w')} - \Q_{\w'} \ast \rho_{\ep(\w)}](\pi_1 \w'+\pi_1 S(t) \wt \u^{0}) \label{wdiff2}\\
&+ [\Q_{\w'} - \Q_{\w'} \ast \rho_{\ep(\w)}](\pi_1 \w - \pi_1 \w') \label{wdiff3}\\
&+ [(\Q_{\w} - \Q_{\w'})\ast \rho_{\ep(\w)}] (\pi_1 S(t) \wt \u^{0} - \pi_1 S(t) \wt \u^{0} \ast \rho_{\ep(\w)}) \label{wdiff4}\\
&+ [\Q_{\w'} \ast \rho_{\ep(\w)} - \Q_{\w'} \ast \rho_{\ep(\w')}](\pi_1 S(t) \wt \u^{0} - \pi_1 S(t) \wt \u^{0} \ast \rho_{\ep(\w)}) \label{wdiff5}\\
&+ [\Q_{\w'} \ast \rho_{\ep(\w')}](\pi_1 S(t) \wt \u^{0}\ast \rho_{\ep(\w')} - \pi_1 S(t) \wt \u^{0} \ast \rho_{\ep(\w)}) \label{wdiff6}.
\end{align}
First of all, notice that 
\begin{align*}
\Q_{\w}-\Q_{\w'}= (3 \pi_1 \u_1 +\pi_1(\w+\w')+2\pi_1 S(t)\wt \u^{0}) \pi_1(\w-\w'),
\end{align*}
so 
\begin{equation} \label{Qdiff2}
\|\Q_{\w}-\Q_{\w'}\|_{W^{\al,\frac 2 {1 -\al}}} \les (\|\u_1\|_{\X} + \| \w\|_{\H^1} +  \| \w'\|_{\H^1} + \|\wt \u_0\|_{\X})\| \w - \w'\|_{\H^1}.
\end{equation}
Therefore, by definition of $\eps$, we have 
\begin{equation}\label{epsdiff}
\begin{aligned}
\|\eps(\w') - \eps(\w)\|^\frac \al 2 \les &~ \| \w - \w'\|_{\H^1}(1 +\|\u_1\|_{\X} + \| \w\|_{\H^1} +  \| \w '\|_{\H^1} + \|\wt \u_0\|_{\X}) \\
&\times (1 + \|\w'\|_{\H^1} + \|\u_1\|_{\X} + \|\wt \u_0\|_{\X} + \|\Q_{\w'}\|_{W^{\al,\frac 2 {1 -\al}}})^{-2}   \\
&\times (1 + \|\w\|_{\H^1} + \|\u_1\|_{\X} + \|\wt \u_0\|_{\X} + \|\Q_{\w}\|_{W^{\al,\frac 2 {1 -\al}}})^{-1}
\end{aligned}
\end{equation}
Therefore, by Lemma \ref{LEM:convdiff}, \eqref{Qdiff2}, and definition of $\eps$, we have
\begin{equation} \label{wdiffest1}
\begin{aligned}
\| \eqref{wdiff1} \|_{L^2}& \les \|(\Q_\w - \Q_{\w'}) - (\Q_\w - \Q_{\w'}) \ast \rho_{\ep(\w)}\|_{L^{\frac2{1-\al}}}\|\pi_1 \w+\pi_1 S(t) \wt \u^{0}\|_{L^\frac2\al} \\
&\les e^{-\frac t {16}}\ep(\w)^{\frac \al 2} \|(\Q_\w - \Q_{\w'}) \|_{W^{\al,\frac 2{1-\al}}} (\|\w\|_{Y} + \|\wt \u^{0}\|_{X^\al})\\
& \les C^{-1} e^{-\frac t {8}} \| \w - \w'\|_{Y},
\end{aligned}
\end{equation}
Similarly, recalling that $f \ast \rho_{\ep_1 + \ep_2} = f \ast \rho_{\ep_1} \ast  \rho_{\ep_2}$ and \eqref{epsdiff}, we have
\begin{equation} \label{wdiffest2}
\begin{aligned}
\| \eqref{wdiff2} \|_{L^2}& \les \|\Q_{\w'} \ast \rho_{\ep(\w')} - \Q_{\w'} \ast \rho_{\ep(\w)}\|_{L^{\frac2{1-\al}}}\|\pi_1 \w'+\pi_1 S(t) \wt \u^{0}\|_{L^\frac2\al} \\
&\les e^{-\frac t {16}}(\ep(\w')-\ep(\w))^{\frac \al 2} \|\Q_{\w'} \|_{W^{\al,\frac 2{1-\al}}} (\|\w'\|_{Y} + \|\wt \u^{0}\|_{X^\al})\\
& \les C^{-1} e^{-\frac t {8}} \| \w - \w'\|_{Y}.
\end{aligned}
\end{equation}
By Lemma \ref{LEM:convdiff}, we have 
\begin{equation} \label{wdiffest3}
\begin{aligned}
\| \eqref{wdiff3} \|_{L^2}& \les \|\Q_{\w'} - \Q_{\w'} \ast \rho_{\ep(\w)}\|_{L^{\frac2{1-\al}}} \|\pi_1 \w-\pi_1 \w'\|_{L^\frac2\al} \\
&\les e^{-\frac t {16}}(\ep(\w))^{\frac \al 2} \|\Q_{\w'} \|_{W^{\al,\frac 2{1-\al}}} \|\w - \w'\|_{Y} \\
&\les e^{-\frac t {16}}(\ep(\w'))^{\frac \al 2} \|\Q_{\w'} \|_{W^{\al,\frac 2{1-\al}}} \|\w - \w'\|_{Y} \\
&\les C^{-1} e^{-\frac t {16}}\|\w - \w'\|_{Y},
\end{aligned}
\end{equation}
Proceeding analogously, 
\begin{equation} \label{wdiffest4}
\begin{aligned}
\| \eqref{wdiff4} \|_{L^2}& \les e^{-\frac t4} \ep(\w)^\frac \al 2\| \Q_\w - \Q_{\w'} \|_{W^{\al,\frac 2{1-\al}}}  \| \wt \u^0\|_{X^\al},\\
& \les e^{-\frac {5t}{16}} C^{-1}  \|\w - \w'\|_{Y},
\end{aligned}
\end{equation}
\begin{equation}\label{wdiffest5}
\begin{aligned}
\| \eqref{wdiff5} \|_{L^2} &\les e^{-\frac t4} \eps(\w)^{\frac 2 \al }(\eps(\w') - \eps(\w))^\frac 2 \al \|\Q_{\w'}\|_{W^{\al,\frac2{1-\al}}} \| \wt \u^0\|_{X^\al}\\
& \les C^{-1} e^{-\frac t4}(\eps(\w') - \eps(\w))^\frac 2 \al \\
& \les e^{-\frac {5t}{16}} C^{-1}  \|\w - \w'\|_{Y},
\end{aligned}
\end{equation}
and finally,
\begin{equation}\label{wdiffest6}
\begin{aligned}
\| \eqref{wdiff6} \|_{L^2} &\les e^{-\frac t4} (\eps(\w') - \eps(\w))^\frac 2 \al \|\Q_{\w'}\|_{W^{\al,\frac2{1-\al}}} \| \wt \u^0\|_{X^\al}\\
& \les e^{-\frac {5t}{16}} C^{-1}  \|\w - \w'\|_{Y}.
\end{aligned}
\end{equation}
By putting together \eqref{wdiffest1}, \eqref{wdiffest2}, \eqref{wdiffest3}, \eqref{wdiffest4}, \eqref{wdiffest5}, \eqref{wdiffest6}, we obtain 
\begin{equation*}
\| F(\w) - F(\w')\|_{L^2} \les e^{-\frac {t}{8}} C^{-1}  \|\w - \w'\|_{Y},
\end{equation*}
so 
\begin{align*}
\phantom{=\ }\| \Gamma[\w](t) - \Gamma[\w'](t)\|_{\H^1}
&\les  \int_{0}^t e^{- \frac{t-t'}{2}} \| F(\w)(t') - F(\w')(t') \|_{L^2} dt' \\
& \les C^{-1}  \int_{0}^t e^{- \frac{t-t'}{2}}  e^{-\frac {t'} 8} \|\w - \w'\|_{Y} dt'\\
&\les e^{-\frac t 8} C^{-1}  \|\w - \w'\|_{Y}.
\end{align*}
Therefore, if $C$ is big enough, we obtain that 
\begin{equation*}
\| \Gamma[\w] - \Gamma[\w']\|_{Y} \le \frac 12  \|\w - \w'\|_{Y}.
\end{equation*}
In particular, by Banach's fixed point theorem, we have that the equation \eqref{weqn} is globally well-posed in $\H^1$. Moreover, since $\Gamma$ is a contraction with $\Lip(\Gamma) \le \frac 12$, we have 
\begin{align*}
\| \w \|_Y &\les \| \Gamma[\0] \|_Y \\
&\les \| \w^0\|_{\H^1} + \sup_t e^{\frac t {16}}\int_0^t e^{-\frac{t-t'}2} \| F(\0)\|_{L^2}\\
&\les \| \w^0\|_{\H^1} + \sup_t e^{\frac t {16}}\int_0^t e^{-\frac{t-t'}2}  e^{- \frac {t'}{4}} \eps(\0)^{\frac \al 2} (\| \Q_\0 \|_{W^{\al, \frac 2{1-\al}}} \| \wt \u^0 \|_{X^\al}) \\
&\les \| \w^0\|_{\H^1} + \| \wt \u^0 \|_{X^\al},
\end{align*}
which proves \eqref{wbound}.

\end{proof}

\begin{proposition}\label{LEM:h} Let $s>0$, $\g \in \R$, $\u_1^{0},\u_2^{0}\in \X$ and $\u_1$ be the unique global solution of \eqref{duhamel} with $\u\vert_{t=0}=\u_{1}^{0}$ and of the form \eqref{u1inw}. Let $\w$ be the unique global solution to \eqref{weqn} with $\w^{0}=0$.
Then, the function 
\begin{align}
\begin{split}
h(t):= \frac{1}{\sqrt{2}}\jb{\nb}^{s}\big[ &\mathcal{Q}(\u_1(t), \w(t)+ S(t)(\u_2^{0}-\u_1^{0})\ast \rho_{\ep(\w)}) \\
&\times(\pi_1 \w(t)+\pi_1 S(t)(\u_2^{0}-\u_1^{0})\ast \rho_{\ep(\w)})\big],
\end{split}\label{h}
\end{align}
is adapted to the filtration $\{\mathcal{F}_{t}\}_{t\geq 0}$. Moreover, $\| h(t) \|_{L^2}$ is a continuous function of $t$, and
\begin{align} 
\E\big[  \|h\|_{L^{2}([0,\infty);L^{2}_{x})}^{2} \big] \le C\|\u_1^{0}-\u_2^{0}\|^2_{\X}<+\infty, \label{hL2}
\end{align}
where $C= C(\|\u_1^{0}\|_{\X}, \|\u_1^{0}-\u_2^{0}\|_{\X})> 0$ is a nondecreasing function of its arguments.
\end{proposition}

\begin{proof}
Since $\w$ and $\u_1$ are adapted to the filtration $\{\mathcal{F}_{t}\}_{t\geq 0}$, it is easy to see from \eqref{h} that $h$ is also adapted. Let 
$\Q_\w = \mathcal{Q}(\u_1(t), \w(t)+ S(t)(\u_2^{0}-\u_1^{0}))$, $\wt \u^0 := \u_2^{0}-\u_1^{0}$. By \eqref{smoothing} with $\ta = s$, we see that $h(t)$ is a continuous function of time with values in $L^2$ and 
\begin{equation}\label{hestimate}
\begin{aligned}
\| h(t) \|_{L^2} &\les \eps(\w)^{-\frac s 2} \|\Q_\w\|_{W^{\al, \frac2{1-\al}}} (\| \w\|_{\H^1} + \|S(t)\wt \u^0\|_{\X}) \\
&\les e^{-\frac t {16}}  (1 +\|\wt \u^0\|_{\X} + \| \u_1\|_{\X} + \| \Q_\w \|_{W^{\al, \frac2{1-\al}}} )^{k(s,\al)} \|\wt \u^0\|_{\X},
\end{aligned}
\end{equation}
where we used \eqref{epsilon} and \eqref{wbound} in the last inequality, and $k(s,\al)$ is an appropriate constant.
Therefore, by \eqref{Qbound}, \eqref{wbound} and \eqref{vmoments},
\begin{align*}
&\phantom{\le\ }\E\bigg[\int_0^{+\infty} \| h(t) \|_{L^2}^2 dt \bigg]\\
&\les \E\bigg[\int_0^{+\infty} e^{-\frac t 8} (1 +\|\wt \u^0\|_{\X} + \| \u_1\|_{\X} + \| \Q_\w \|_{W^{\al, \frac2{1-\al}}} )^{2k(s,\al)} \|\wt \u^0\|_{\X}^2 dt \bigg]\\
&\les \|\wt \u^0\|_{\X}^2 \int_0^{+\infty} e^{-\frac t 8}  \Big[ 1 + \| \u_1^0\|_{\X}^2 + \|\wt \u^0\|_{\X}^2 + \eta(t;\|\u^1_0\|_{\X})^{k(s,\al)} \Big] d t\\
&\le C(\|\u_1^{0}\|_{\X}, \|\u_1^{0}-\u_2^{0}\|_{\X})\|\wt \u^0\|_{\X}^2.
\end{align*}
\end{proof}

For $\u_1^0, \u_2^0 \in \X$, we define 
\begin{align*}
\u_2^{h}:=\w+\u_1+S(t)(\u^0_2-\u^0_1),
\end{align*}
where $\u_1(t) = \Phi_t(\u_1^0; \xi)$, and $\w$ is the solution to \eqref{weqn} with $\w^0 = \0$. 
For $h$ as in \eqref{h}, comparing \eqref{weqn} with \eqref{veqn} we obtain that 
\begin{align} \label{u2h}
\u_2^{h} = \Phi_{t}(\u_2^0, \xi+h)=\Phi_{t}(\u_1^0,\xi)+S(t)(\u_2^0-\u_1^0)+\w(t).
\end{align}

\begin{remark}\rm
Note that if $L_t$ is the linear flow; namely, $L_{t}(\u_0;\xi)=S(t)\u_0+\stick_{t}(\xi)$, then we can choose $h=0$ and $\w=0$ and obtain the analogous version of \eqref{u2h}
\begin{equation*}
L_{t}(\u_2^0;\xi) = L_t(\u_1^0;\xi) + S(t)(\u_2^0-\u_1^0).
\end{equation*}
We then have for any $\mu_1$ and $\mu_2$ invariant, 
\begin{align*}
\|\mu_1 -\mu_2\|_{d_{n}}&\leq \int \E[ d_{n}(S(t)(\u_2^0-\u_1^0), {\bf{0}})] d\mu_1(\u_1^0) d \mu_2(\u_2^0)\\
& \leq \int \big(ne^{-\frac t4} \| \u_2^0 - \u_1^0\|_{X^\al} \wedge 1\big) d\mu_1(\u_1^0) d \mu_2(\u_2^0) \to 0
\end{align*}
as $t \to \infty$, which implies that $\|\mu_1 -\mu_2\|_{d_{n}} = 0$ for every $n \in \N$. In particular, we have that $\mu_1 = \mu_2$, i.e.\ uniqueness of the invariant measure.
\end{remark}

\subsection{Proof of Theorem~\ref{THM:erg} and Theorem~\ref{THM:conv}} \label{SEC:proofThm1}
Given $n\in \N$, recall that  
\begin{align}
d_{n}(\mathbf{x},\mathbf{y})=1\wedge n\|\mathbf{x}-\mathbf{y}\|_{\X}, \label{dn1}
\end{align}
and denote by $\text{Lip}_{d_n, \frac{1}{2}}$ the set of $d_{n}$-Lipschitz continuous functions $F:\X\to \R$ such that $\|F\|_{d_{n}}\leq 1$ and $\|F\|_{L^{\infty}}\leq \frac{1}{2}$.

\begin{proposition}\label{PROP:flowclose} Let $n\in \N$ and $R>0$ be fixed. Given any $\u_1^{0},\u_2^{0}\in B_{R}( \bf{0})$, there exist $t_0=t_0(n,R)>0$ and $K(R)>0, K'(R) > 0$ such that, for every $t \ge t_0$,
\begin{align}
\E\big[ F(\Phi_{t}(\u_1^{0},\xi))\big] -\E\big[ F(\Phi_{t}(\u_2^{0},\xi))\big] \leq (1-K(R))\big(1 \wedge K'(R) \| \u_1^0-\u_2^0\|_{\X}^\frac12\big), \label{diffflow}
\end{align}
for any $F\in \textup{Lip}_{d_n ,\frac{1}{2}}$.
\end{proposition}

We postpone the proof of this proposition to the end of the section, and proceed to the proofs of Theorem \ref{THM:erg} and Theorem \ref{THM:conv}.

\begin{proof}[Proof of Theorem~\ref{THM:erg}]
Suppose, in order to obtain a contradiction, that there exist probability measures $\mu_1$ and $\mu_2$, $\mu_{1}\neq \mu_2$, which are both invariant under the flow of \eqref{duhamel}. We first suppose that $\mu_1 \perp \mu_2$. We will return to the case when they are not mutually singular at the end of the proof. 
We fix $n\in \N$. 
Let $R>0$ be large enough so that $\displaystyle{\min_{j=1,2}} \mu_{j,R}>\frac{1}{2}$, where $\mu_{j,R}:=\mu_{j}(B_{R}(\bf{0}))$ and let $t=t(n,R)>0$ and $K(R)>0$ be given by Proposition~\ref{PROP:flowclose}.
By Lemma~\ref{LEM:MKDuality}, Lemma~\ref{LEM:testLinfty} 
and the invariance of $\mu_1$ and $\mu_2$, we have 
\begin{align*}
\| \mu_1-\mu_2\|_{d_{n}}& = \sup_{F\in \text{Lip}_{d_n ,\frac{1}{2}}} \bigg[ \int_{\X} F(\u_1^{0})d\mu_1 (\u_1^{0})-\int_{\X} F(\u_2^{0})d\mu_2 (\u_2^{0}) \bigg] \\ 
& =\sup_{F\in \text{Lip}_{d_n ,\frac{1}{2}}}  \iintt_{\X \times \X}\E\big[ F(\Phi_{t}(\u_1^{0},\xi))\big] -\E\big[ F(\Phi_{t}(\u_2^{0},\xi))\big] d\mu_1( \u_1^{0}) d\mu_2(\u_2^{0}).
\end{align*}
Using that $F\in \text{Lip}_{d_n, \frac{1}{2}}$ and Proposition~\ref{PROP:flowclose}, we have  
\begin{align*}
\| \mu_1 -\mu_2\|_{d_n} & \leq \sup_{F\in \text{Lip}_{d_n ,\frac{1}{2}}} \bigg\vert \iintt_{B_{R}({\bf{0}}) \times B_{R}({\bf{0}})}\E\big[ F(\Phi_{t}(\u_1^{0},\xi))\big] -\E\big[ F(\Phi_{t}(\u_2^{0},\xi))\big] d\mu_1( \u_1^{0}) d\mu_2(\u_2^{0}) \bigg\vert \\
& \hphantom{XX} + \mu_{1,R}(1-\mu_{2,R})+(1-\mu_{1,R})\mu_{2,R} +(1-\mu_{1,R})(1-\mu_{2,R}) \\
&\leq (1-K(R)) \mu_{1,R}\mu_{2,R}+1-\mu_{1,R}\mu_{2,R} \\
& = 1-K(R)\mu_{1,R}\mu_{2,R} \\
& \leq 1-\frac{1}{4}K(R).
\end{align*}
Note that the right hand side of the above inequality is independent of $n\in \N$. 
Now, taking $n\to \infty$ using Lemma~\ref{LEM:dnlimit}, we have shown 
\begin{align*}
\| \mu_1-\mu_2\|_{\text{TV}} \leq 1-\frac{1}{4}K(R) <1,
\end{align*}
which is a contradiction to the fact that $\mu_1 \perp \mu_2$. Thus, $\mu_1 =\mu_2$.
If instead $\mu_1$ and $\mu_2$ are not mutually singular, then we consider the measures
\begin{align*}
\rho_{1}&=\frac{1}{(\mu_1-\mu_2)_{+}(\X)}(\mu_1-\mu_2)_{+},\\
\rho_{2}& = \frac{1}{(\mu_2-\mu_1)_{+}(\X)}(\mu_2-\mu_1)_{+}.
\end{align*}
These are probability measures which are invariant under the flow of \eqref{duhamel} and moreover, $\rho_{1}\perp \rho_{2}$. By the same argument as above with $(\rho_1,\rho_2)$ replacing $(\mu_1,\mu_2)$ we would obtain $\| \rho_1 -\rho_2 \|_{\text{TV}}<1$, which is a contradiction. Hence, we again conclude that $\mu_1 =\mu_2$.
\end{proof}

\begin{proof}[Proof of Theorem~\ref{THM:conv}]
Fix $\eps > 0$, and let $\xi, \xi'$ be two independent copies of space-time white noise. Let $\u_1^0, \u_2^0 \in \X$, and fix $R \gg 1$. We want to show that the stopping time 
\begin{equation*}
\tau(\u_1^0, \u_2^0) = \min\{ t: \Phi_t(\u_0, \xi), \Phi_t(\u_1^0, \xi') \in B_R(\0), \|\Phi_t(\u_0, \xi) - \Phi_t(\u_1^0, \xi')\|_{\X}>\eps \} 
\end{equation*}
is finite a.s.
Let $\nu_t = \nu_t(\u_1^0, \u_2^0)$ be the law of the process 
$(\Phi_t(\u_1^0, \xi), \Phi_t(\u_2^0, \xi'))$, which takes values in $\X \times \X$.
By Theorem \ref{THM:erg} and Proposition \ref{PROP:KryBo}, we have that 
\begin{equation*}
\frac 1 T \int_0^T \nu_t  dt \rightharpoonup \rho_s \otimes \rho_s
\end{equation*} 
weakly.
For $R \gg 1$, let $f_{R,\eps}: \X \times \X \to \R$ be given by
\begin{equation*}
f_{R,\eps}(x,y) = \Big(1 - 1 \wedge \frac{\|x-y\|_{\X}}\eps\Big) \Big(1 - 1 \wedge \frac{\|x\|_{\X}}{R}\Big)\Big(1- 1 \wedge \frac{\|y\|_{\X}}{R}\Big).
\end{equation*}
Therefore, we have that 
\begin{equation} \label{limitKryBo}
\lim_{T \to \infty}\frac 1 T \int_0^T \E[f_{R,\eps}(\Phi_t(\u_1^0, \xi), \Phi_t(\u_2^0, \xi'))]  dt = \iint f_{R,\eps}(x,y) d\rho_s(x) d\rho_s(y) = 2\gamma_{R,\eps}.
\end{equation}
Suppose that $\gamma_{R,\eps} > 0$. Since $f_{R,\eps} \le 1$ and $f_{R,\eps}(\Phi_t(\u_1^0, \xi), \Phi_t(\u_2^0, \xi')) = 0$ for $t < \tau$,
by \eqref{limitKryBo} we can deduce that there exists $T(\u_1^0, \u_2^0)$ such that 
\begin{equation*}
\prob(\{ \tau(\u_1^0, \u_2^0) \le T(\u_1^0, \u_2^0)\}) \ge  \gamma_{R,\eps}.
\end{equation*}
Define recursively the stopping time $\tau_k$ with $\tau_0 = T(\u_1^0, \u_2^0)$, and 
\begin{equation*}
\tau_{k+1} = \tau_{k} + T(\Phi_{\tau_k}(\u_1^0, \xi), \Phi_{\tau_k}(\u_2^0, \xi')).
\end{equation*}
We have that
\begin{align*}
&\phantom{=\ }\prob\{\tau > \tau_{k}\} \\
&= \prob\{\tau(\Phi_{\tau_{k-1}}(\u_1^0, \xi), \Phi_{\tau_{k-1}}(\u_2^0, \xi')) > T(\Phi_{\tau_k}(\u_1^0, \xi), \Phi_{\tau_{k-1}}(\u_2^0, \xi'))| \tau > \tau_{k-1}\} \prob\{ \tau > \tau_{k-1} \}\\
&\le (1-\gamma_{R,\eps}) \prob\{ \tau > \tau_{k-1} \}\\
& \le (1-\gamma_{R,\eps})^{k+1} \to 0 \text{ as }k \to \infty.
\end{align*}
where we used the (strong) Markov property for the penultimate inequality, and a simple induction for the last inequality. Therefore, $\tau$ is finite a.s., as long as $\gamma_{R,\eps} > 0$. We notice that 
\begin{equation*}
f_{R,\eps} \ge \frac 1 8 \quad  \text{ for } \quad \|x-y\|_{\X} \le \frac \eps 2,\, \|x\|_{\X}  \le \frac R2, \, \|y\|_{\X}  \le \frac R2.
\end{equation*}
Let $R$ be such that $\rho_s(B_{\frac R 4}(\0)) > 0$, and let $x_0 \in B_{\frac R 4}(\0)$ be such that $\rho_s(B_{\frac \eps 4}(x_0)) > 0$. Such an $x_0$ must exist because countably many such balls cover $ B_{\frac R 4}(\0)$. Therefore,
\begin{align*}
\gamma_{R,\eps} &=  \frac 12 \iint f_{R,\eps}(x,y) d\rho_s(x) d\rho_s(y) \\
&\ge \frac 1 {16} \iint \ind_{\{ \|x-y\|_{\X}  \le \frac \eps 2, \|x\|_{\X}  \le \frac R2,  \|y\|_{\X}  \le \frac R2\}}(x,y) d\rho_s(x) d\rho_s(y)\\
&\ge  \frac 1 {16} \rho_s(B_{\frac \eps 4}(x_0))^2 > 0.
\end{align*}
We now move to proving \eqref{EQN:WassConv}. Fix $\u_0 \in \X$, and let $F \in \Lip_{d_n,\frac12}$. By Proposition \ref{PROP:flowclose} and the (strong) Markov property, we have that 
\begin{align*}
&\phantom{\le\ }\int \E\big[ F(\Phi_{t}(\u_0,\xi))\big] -\E\big[ F(\Phi_{t}(\u_1^{0},\xi'))\big] d \rho_s(\u_1^0)\\
&\le \int \prob(\{ \tau(\u_0,\u_1^0) > t - t_0(n,R) \}) d \rho_s(\u_1^0) \\
&\phantom{\le}+ \int \E\big[ \ind_{\{ \tau(\u_0,\u_1^0) \le t - t_0(n,R) \}}F(\Phi_{t - \tau}(\Phi_\tau(\u_0,\xi), \xi'')) -  F(\Phi_{t - \tau}(\Phi_\tau(\u_1^0,\xi'), \xi'''))\big] d \rho_s(\u_1^0) \\
&= \int \prob(\{ \tau(\u_0,\u_1^0) > t - t_0(n,R) \}) d \rho_s(\u_1^0) \\
&\phantom{\le}+  \int \E_{\xi,\xi'}\big[\ind_{\{ \tau(\u_0,\u_1^0) \le t - t_0(n,R) \}}\Big(\E_{\xi''}[F(\Phi_{t - \tau}(\Phi_\tau(\u_0,\xi), \xi''))\big] \\
&\phantom{\le + \int \E_{\xi,\xi'}\Big[\ind_{\{ \tau(\u_0,\u_1^0) \le t - t_0(n,R) \}}]\ }-  \E_{\xi'''}[F(\Phi_{t - \tau}(\Phi_\tau(\u_1^0,\xi'), \xi'''))\Big)\Big] d \rho_s(\u_1^0) \\
&\le \int \prob(\{ \tau(\u_0,\u_1^0) > t - t_0(n,R) \}) d \rho_s(\u_1^0)+ K'(R) \eps^\frac12,
\end{align*}
where $\xi'', \xi'''$ are copies of space-time white noise (that correspond to $\xi(t - \tau), \xi'(t-\tau)$ respectively). Therefore, by Lemma \ref{LEM:testLinfty}, and since $\tau$ is finite a.s., by dominated convergence we have 
\begin{align*}
\limsup_{t \to \infty} \| \Pt{t}^\ast \dl_{\u_0} - \rho_s \|_{d_n} & \le \limsup_{t \to \infty} \int \prob(\{ \tau(\u_0,\u_1^0) > t - t_0(n,R) \}) d \rho_s(\u_1^0) + K'(R) \eps^\frac12\\
&\le K'(R) \eps^\frac12.
\end{align*}
Since $\eps>0$ was arbitrary, we deduce \eqref{EQN:WassConv}. 
\end{proof}

Before proceeding to the proof of Theorem \ref{PROP:flowclose}, we need the following estimate.
\begin{lemma}\label{LEM:prob}
Let $X\geq 0$ be a random variable such that $\E[X]=1$. Then, for any $0\le \ta \le 1$, we have 
\begin{align*}
\E \big[ \big| X-1 \big| \big] \leq 2\big[  1-\ta \,\prob\big(X\geq \ta \big) \big]
\end{align*}
\end{lemma}
\begin{proof}
We write 
\begin{align*}
\E \big[ \big| X-1 \big| \big]  & = \E \big[ ( X-1)\ind_{\{ X>1\}} \big]+\E \big[ ( 1-X)\ind_{\{ X\leq 1\}} \big] \\
& = \E\big[ X\ind_{\{ X>1\}}\big]-\prob(X>1)+\E \big[ ( 1-X)\ind_{\{ X\leq 1\}} \big] \\
& =1-\E \big[ X\ind_{\{ X\leq 1\}} \big] -\prob(X> 1)+\E \big[ ( 1-X)\ind_{\{ X\leq 1\}} \big] \\
& = 2\big(1-\E \big[ X\ind_{\{ X\leq 1\}} \big]-\prob(X> 1)\big).
\end{align*}
Using that 
\begin{align*}
\E\big[ X\ind_{\{X\leq 1\}}\big] & \geq \E\big[ X \ind_{\{X\geq \ta\}\cap \{X\leq 1\}}  \big] \geq \ta \prob\big( \{X\geq \ta\} \cap \{ X\leq 1\}\big)
\end{align*}
and that $0\le\ta\le1$ implies 
\begin{align*}
\prob( X>1) \geq \ta \prob(X>1)= \ta \prob\big( \{ X>1\}\cap \{X\geq \ta\} \big),
\end{align*}
we therefore have 
\begin{align*}
\E \big[ \big| X-1 \big| \big] &\leq 2\big[ 1-\ta \prob\big( \{X\geq \ta\} \cap \{ X\leq 1\}\big)-\ta \prob\big( \{ X>1\}\cap \{X\geq \ta\} \big)\big] \\
& = 2\big[ 1-\ta \prob\big( X\geq \ta\big)\big].
\end{align*}
\end{proof}

\begin{remark}\rm \label{RMK:prob}
We will make use of the following consequence of Lemma~\ref{LEM:prob}: Let $ 1 \le p < \infty$, and let $X$ be a random variable such that $\E[ e^{X}]=1$ and $\E[|X|^{p}] < +\infty$. Then, for any $L>0$, we have
\begin{align} \label{TVest}
\E\big[ \big| e^{X}-1  \big| \big] \leq 2\bigg( 1-e^{-L}+e^{-L}L^{-p}\E\big[ |X|^{p}\big]  \bigg)
\end{align}
This follows from Lemma~\ref{LEM:prob} by putting $\ta=e^{-L}$ and applying Chebyshev's inequality to
\begin{align*}
\prob\big( e^{X}\geq e^{-L}\big)  \geq \prob\big(L\geq \log e^X \geq -L\big) = \prob\big( |X|\leq L\big) = 1-  \prob\big( |X|> L\big).
\end{align*}
\end{remark}

\begin{proof}[Proof of Proposition~\ref{PROP:flowclose}]

Fix $\u_1^{0},\u_2^{0}\in \X $ and $t>0$.
For $M\gg 1$ to be chosen later, let
 \begin{align*}
\tau_{M}&=\tau_{M}(\xi) \\ 
&=\inf_{t>0}\{   \max \big\{   \|\stick_{t}(\xi)\|_{C([0,t]; \mathcal{W}^{\al,\frac{4}{\al}})}, \|\wick{\stick_{t}(\xi)^{2}}_{\g}\|_{C([0,t];L^{4}_{x})}, \|\wick{\stick_{t}(\xi)^{3}}_{\g}\|_{C([0,t];L^{2}_{x})}   \big\}  > M\}
\end{align*}
 Then, $\tau_{M}$ is a stopping time. 
 Moreover, by Proposition \ref{PROP:REGSTICK}, we have that $\tau_M \to \infty$ a.s.\ as $M \to \infty$.
Proposition~\ref{PROP:GWP} implies that the processes 
\begin{align*}
\Phi_{t}(\u_{j}^{0},\xi)=S(t)\u_j ^{0}+\stick_{t}(\xi)+\v_{j}(t; \u_j ^{0},\xi), \quad j=1,2,
\end{align*}
are a.s.\ well-defined, where $\v_{j}(t; \u_j ^{0},\xi)$ are the unique solutions to \eqref{veqnd}, for $j=1,2$, and moreover there exists $C_1=C_1(M,R)>0$ such that 
\begin{align}
\max_{j=1,2} \|\v_{j}(t;\u_{j}^{0},\xi)\|_{C([0,t\wedge \tau_{M}]; \H^1)} \leq C_1. \label{vbd}
\end{align}
By Proposition~\ref{PROP:shift}, we can define $\w\in C([0,\infty); \H^{1})$ as the unique solution to \eqref{weqn} with $\w^{0}=0$ and with respect to the stochastic flow $\Phi_{1}(t;\u_1 ^{0},\xi)$, and $\w$ satisfies 
\begin{align}
\| \w(t) \|_{\H^1} \les e^{-ct}\|\u_2^{0}-\u_1^{0}\|_{\X}, \label{wbd}
\end{align}
where the implicit constant is independent of $t,R,$ and $M$.
From Lemma~\ref{LEM:h}, we form the function $h$ according to \eqref{h} and denote by $h_M$ the process $h$ stopped at time $\tau_{M}$, that is, $h_{M}(t):=h(t\wedge \tau_M)$. We have that $h_{M}$ is adapted, and by \eqref{hestimate} and \eqref{vbd}, it satisfies 
\begin{align*}
\|h_{M}(t)\|_{L^{2}_{x}} & \leq C_{2}(M,R).  
\end{align*}
This implies that $h_M$ satisfies the Novikov condition
\begin{align*}
\E \bigg[ \exp\bigg(\frac{1}{2}\int_{0}^{t} \|h_{M}(t')\|_{L^2_x}^{2} dt' \bigg) \bigg] \leq e^{tC^{2}_{2}(M,R)}<+\infty
\end{align*}
for any $t>0$.
For $f\in L^{2}_{t,x}$, we define 
\begin{align*}
\EE(f):= \exp\bigg( -\frac{1}{2}\int_{0}^{t} \|f(t')\|_{L^2_{x}}^{2}dt'+\int_{0}^{t} \jb{ f(t') ,\xi(t')}_{L^2_{x}}dt'  \bigg).
\end{align*}
By \eqref{u2h}, we have that
\begin{align}
\Phi_{t\wedge \tau_M}(\u_2^{0},\xi+h_M)=\Phi_{t\wedge \tau_M}(\u_1 ^{0},\xi)+\mathbf{r}_M (t),\label{flowdiff}
\end{align}
where 
\begin{align*}
{\bf r}_{M}(t):=S(t\wedge \tau_M)(\u_2 ^{0}-\u_1 ^{0})+\w(t\wedge \tau_M).
\end{align*}
We have
\begin{align}
&\E \big[ F(\Phi_{t}(\u_1^{0},\xi))-F(\Phi_{t}(\u_2^{0},\xi))   \big] \notag\\
& = \E \big[ \ind_{\{t\leq \tau_{M}(\xi) \}}\big\{  F( \Phi_{t\wedge \tau_M(\xi)}(\u_1^{0};\xi) )-F( \Phi_{t\wedge \tau_M(\xi)}(\u_2^{0};\xi) )   \big\} \big]  \label{exp1}\\
& \hphantom{XXX}+ \E \big[ \ind_{\{t >\tau_{M}(\xi) \}}\big\{  F( \Phi_{t}(\u_1^{0};\xi) )-F( \Phi_{t}(\u_2^{0};\xi) )   \big\} \big].  \label{exp2}
\end{align}
Since $F\in \text{Lip}_{d_n,\frac{1}{2}}$, we have 
\begin{align}
|\eqref{exp2}| \leq 2\|F\|_{L^\infty}\prob( t>\tau_{M}(\xi)). \label{exp2bd}
\end{align}
Using \eqref{flowdiff}, we have 
\begin{align}
&\eqref{exp1}  \notag\\
&=\E \big[ \ind_{\{t\leq \tau_{M}(\xi) \}}\big\{  F( \Phi_{t\wedge \tau_M(\xi)}(\u_2^{0};\xi+h_M)-{\bf r}_{M}(t) )-F( \Phi_{t\wedge \tau_M(\xi)}(\u_2^{0};\xi+h_M) )   \big\} \big]  \label{exp3}\\
& \hphantom{XX} +\E \big[ \ind_{\{t\leq \tau_{M}(\xi) \}}\big\{  F( \Phi_{t\wedge \tau_M(\xi)}(\u_2^{0};\xi+h_M) )-F( \Phi_{t\wedge \tau_M(\xi)}(\u_2^{0};\xi) )   \big\} \big]. \label{exp4}
\end{align}
Using \eqref{lipd}, \eqref{dn1}, \eqref{S(t)onX} and \eqref{wbd}, we have 
\begin{align}
|\eqref{exp3}| &\leq \|F\|_{d_n} \E \big[  \ind_{\{t\leq \tau_{M}(\xi)\}} d_{n}(S(t)(\u_2^{0}-\u_1^{0}), {\bf 0}) \big] +\E \big[  \ind_{\{t\leq \tau_{M}(\xi)\}} d_{n}(\w(t), {\bf 0}) \big] \notag\\
& \leq 2\|F\|_{d_n} \big( 1 \wedge Cne^{-ct}\|\u_2^{0}-\u_1^{0}\|_{\X}\big). \label{exp3bd}
\end{align}
Meanwhile, by Girsanov's theorem,
\begin{align}
\eqref{exp4} &= \E \big[  \ind_{\{t\leq \tau_{M}(\xi)\}} F( \Phi_{t\wedge \tau_{M}(\xi)}(\u_2^{0},\xi+h_M))\big]-\E \big[\ind_{\{t\leq \tau_{M}(\xi) \}}F( \Phi_{t\wedge \tau_M(\xi)}(\u_2^{0};\xi) )   \big] \notag \\
&= \E \big[  \ind_{\{t\leq \tau_{M}(\xi)\}} F( \Phi_{t}(\u_2^{0},\xi+h_M))\big]-\E \big[\ind_{\{t\leq \tau_{M}(\xi) \}}F( \Phi_{t}(\u_2^{0};\xi) )   \big]  \notag \\ 
&=  \E \big[   \big( F( \Phi_{t}(\u_2^{0},\xi+h_M))  -F( \Phi_{t}(\u_2^{0};\xi) ) \big)  \big] \notag \\
& \hphantom{X} -  \E \big[  \ind_{\{t> \tau_{M}(\xi)\}} \big( F( \Phi_{t}(\u_2^{0},\xi+h_M))  -F( \Phi_{t}(\u_2^{0};\xi) ) \big)  \big] \notag \\
& = \E \big[ F(\Phi_{t}(\u_2^{0},\xi))\big( \EE(h_M) -1\big)  \big]  \label{exp5} \\
& \hphantom{X}-\E \big[  \ind_{\{t> \tau_{M}(\xi)\}} \big( F( \Phi_{t}(\u_2^{0},\xi+h_M))  -F( \Phi_{t}(\u_2^{0};\xi) ) \big)  \big] \label{exp6} 
\end{align}
We bound the second term by
\begin{align}
|\eqref{exp6}| &\leq 2 \|F\|_{L^{\infty}} \prob(t>\tau_M(\xi)). \label{exp6bd}
\end{align}
We now consider \eqref{exp5}. For $L>0$ to be chosen later, let $E:=\{ \EE(h_M)\geq e^{-L}\}$. Since $\E[ \EE(h_M)]=1$, \eqref{TVest}, Ito's isometry and \eqref{hL2} imply 
\begin{align}
|\eqref{exp5}|& \leq 2\|F\|_{L^{\infty}}\big[ 1-e^{-L} + e^{-L}L^{-1}\E[|\log \EE(h_M)|]\big] \notag\\
&\leq 2\|F\|_{L^{\infty}}\bigg[ 1-e^{-L} +e^{-L}L^{-1}\E\big[\|h\|_{L^{2}_{t,x}}^{2} + \|h\|_{L^{2}_{t,x}}\ \big] \bigg] \notag\\
&\leq 2\|F\|_{L^{\infty}}\bigg[ 1 - (1 -C(R)\|\u_1^0 - \u_2^0\|_{\X}L^{-1})e^{-L} \bigg] \label{exp5bd}
\end{align}
Finally, combining \eqref{exp2bd}, \eqref{exp3bd}, \eqref{exp6bd} and \eqref{exp5bd}, we have
\begin{align*}
&~\E \big[ F(\Phi_{t}(\u_1^{0},\xi))-F(\Phi_{t}(\u_2^{0},\xi))   \big] \notag\\ 
\leq&~2\|F\|_{L^{\infty}}\big( 1 - (1 - C(R)\|\u_1^0 - \u_2^0\|_{\X}L^{-1})e^{-L} \big) + 4\|F\|_{L^{\infty}} \prob(t>\tau_M(\xi))\\
&+ 2\|F\|_{d_n} \big( 1 \wedge Cne^{-ct}\|\u_2^{0}-\u_1^{0}\|_{\X}\big).
\end{align*}
Using that $F\in \text{Lip}_{d_n, \frac{1}{2}}$, we get 
\begin{equation}
\begin{aligned}
&\phantom{\le\ }\E \big[ F(\Phi_{t}(\u_1^{0},\xi))-F(\Phi_{t}(\u_2^{0},\xi))   \big]\\
 &\le 1 - (1 -C(R)\|\u_1^0 - \u_2^0\|_{\X}L^{-1})e^{-L} + 2  \prob(t>\tau_M(\xi))+ 2(1 \wedge Cne^{-ct}\|\u_2^{0}-\u_1^{0}\|_{\X})
\end{aligned}\label{diffbdp}
\end{equation}
Now, by Chebyshev's inequality and \eqref{stickC01}, 
\begin{align*}
\prob(t>\tau_M(\xi))&\leq \frac{\E\bigg[ \max \big\{   \|\stick_{t}(\xi)\|_{C([0,t]; \mathcal{W}^{\al,\frac{4}{\al}})}^{2}, \|\wick{\stick_{t}(\xi)^{2}}_{\g}\|_{C([0,t];L^4_{x})}^{2}, \|\wick{\stick_{t}(\xi)^{3}}_{\g}\|_{C([0,t];L^{2}_{x})}^{2}   \big\} \bigg]}{M^{2}} \\
& \leq \frac{C(t)}{M^{2}}.
\end{align*}
Inserting this bound into \eqref{diffbdp} yields 
\begin{align}
&\phantom{\le\ }\E \big[ F(\Phi_{t}(\u_1^{0},\xi))-F(\Phi_{t}(\u_2^{0},\xi))   \big] \notag \\
& \leq 1 - (1 -C(R)\|\u_1^0 - \u_2^0\|_{\X}L^{-1})e^{-L} +  \frac{2C(t)}{M^{2}}+ 2(1 \wedge Cne^{-ct}\|\u_2^{0}-\u_1^{0}\|_{\X}). \label{tobesplit}
\end{align}
We first choose $L=L(R)\gg1$, so that $2C(R)RL^{-1}<\frac{1}{4}$. 
Next, we choose $t_R=t_R(n,R,L)=t_R(n,R)>0$ so that 
\begin{align}\label{tchoice}
2(1 \wedge Cne^{-ct_R}\|\u_2^{0}-\u_1^{0}\|_{\X}) \leq  2(1 \wedge 2CRne^{-ct_R}) \leq \frac{1}{8}e^{-L}.
\end{align}
Then, we choose $M=M(R,t,L)\gg1$ so that 
\begin{align*}
\frac{2C(t)}{M^{2}}<\frac{1}{2}e^{-L}.
\end{align*}
Hence we have, for every $t \ge t_R$, 
\begin{align*}
\E \big[ F(\Phi_{t}(\u_1^{0},\xi))-F(\Phi_{t}(\u_2^{0},\xi))   \big]  & \leq 1- \frac14 e^{-L}+2\big( 1 \wedge Cne^{-ct}\|\u_2^{0}-\u_1^{0}\|_{\X}\big).
\end{align*}
Therefore,
\begin{align}
\eqref{tobesplit} \le 1 - \frac{1}{8} e^{- L(R)} =: 1 - K(R).\label{diffflow1}
\end{align}
Let now $L = R^{-\frac12} \| \u_1^0 - \u_2^0\|_{\X}^\frac 12$. Then, let $t_1 =  t_1(C, n)$ be such that 
\begin{align*}
2(1 \wedge Cne^{-ct_1}\|\u_2^{0}-\u_1^{0}\|_{\X}) \leq  \|\u_2^{0}-\u_1^{0}\|_{\X} \le \|\u_2^{0}-\u_1^{0}\|_{\X}^\frac12,
\end{align*}
and for $t \ge t_1$, let $M=M(R,t, \u_1^0, \u_2^0)\gg1$ so that 
\begin{align*}
\frac{2C(t)}{M^{2}}\le \|\u_2^{0}-\u_1^{0}\|_{\X}^\frac12.
\end{align*}
By our choice of $L$, we have that $L \le \sqrt2$, so  $e^{-L} \ge 1- L = 1 - R^{-1} \|\u_2^{0}-\u_1^{0}\|_{\X}^\frac12$. We obtain that, for some constant $C'(R)$,
\begin{align}
\eqref{tobesplit} \le C'(R) \|\u_2^{0}-\u_1^{0}\|_{\X}^\frac12. \label{diffflow2}
\end{align}
Therefore, \eqref{diffflow} follows from \eqref{diffflow1} and \eqref{diffflow2}, with $K'(R)=\frac{C'(R)}{1-K(R)}$.
\end{proof} 
\begin{remark}\label{convSpeed}
We would like provide some heuristic about our claim that the estimates in this paper can provide at best a sub-polynomial rate of convergence. 
Fix $\eps > 0$. We want to estimate the first time $t$ such that 
\begin{equation*}
\| \Pt{t}^\ast\dl_{\u_1^0} - \rho_s\|_{d_1} \les \eps.
\end{equation*}
Since the estimate \eqref{diffflow} does not provide any information about what happens outside of a ball with radius $R$, we need to fix $R$ with $\rho_s(B_{R}({\bf 0})) \les \eps$. In view of Proposition \ref{PROP:GWP}, this requires us to take $R \gtrsim |\log \eps|^A$, for some $A > 1$.
We notice that by \eqref{tchoice}, we have
\begin{equation*}
t_0 \gtrsim L \gtrsim R C(R).
\end{equation*}
The quantity $C(R)$ in this inequality comes from \eqref{exp5bd}, which in turn depends on \eqref{hL2}. From the proof of Proposition \ref{LEM:h}, we see that the constant $C$ in \eqref{hL2} satisfies 
\begin{equation*}
C \gtrsim \int_0^{+\infty} \eta(t,R)^{k(s,\alpha)} dt \gtrsim \exp(cR),
\end{equation*}
for some constant $c > 0$, and a similar estimate holds for the quantity $C(R)$. Recalling that $R \gtrsim |\log \eps|^A$, we obtain that 
$$t_0 \gtrsim \exp(c'R) \gtrsim \exp(c' |\log \eps|^A), $$
which is a function of $\eps^{-1}$ with super-polynomial growth.
\end{remark}

\begin{remark}\rm \label{RMK:t18erg}

In this remark, we let $\Phi_t$ denote the global-in-time flow of the equation \eqref{2DSDNLW} defined $\rho$ a.s., for $\rho$ in \eqref{Gibbs}; see \cite{t18erg}. Then, we have $\Phi_{t}(\u_1^0;\xi)=S(t)\u_1^0+\stick_{t}(\xi)+\v (t; \u_1^0, \xi)$, where $S(t)$ is the solution operator for the linear homogeneous equation associated to \eqref{2DSDNLW}, $\stick_{t}(\xi)$ is the corresponding stochastic convolution
\begin{align*}
\stick_t(\xi) =  \int_0^t S(t-t')\vec{0}{\sqrt{2}\xi(t')} dt',
\end{align*}
and $\v$ solves
\begin{align*}
\v(t) = - \int_0^t S(t-t') \vec{0}{\pi_1( S(t')\u_1^0+\stick_{t'}(\xi)+\v(t'))^3  } dt'.
\end{align*}
With $\w=0$ and
\begin{align*}
h= \pi_1( \Phi_{t}(\u_1^0;\xi)+S(t)(\u_2^0-\u_1^0))^3- \pi_1( \Phi_{t}(\u_1^0;\xi))^3,
\end{align*}
we obtain 
\begin{align*}
\Phi_t( \u_2^0,\xi+h)=\Phi_{t}(\u_1^0,\xi)+S(t)( \u_2^0-\u_1^0)
\end{align*}
as the analogue of \eqref{u2h}. This choice of $h$ satisfies a corresponding estimate as in \eqref{hL2} since samples of the Gibbs measure $\rho$ and $\stick_{t}(\xi)$ belong to $\mathcal{C}^{\frac{s}{2}}(\T^2)$ a.s. and $\v(t)\in \H^{1}(\T^2)$. 
The uniqueness of the Gibbs measure \eqref{Gibbs} under the flow of \eqref{2DSDNLW} then follows from the arguments in the proofs of Proposition~\ref{PROP:flowclose} and Theorem~\ref{THM:erg} along with good long time bounds for $\v$ which follow from the arguments in~\cite{t18erg}.
\end{remark}

\begin{ackno}\rm
J.\,F.~was supported by The Maxwell Institute Graduate School in Analysis and its
Applications, a Centre for Doctoral Training funded by the UK Engineering and Physical
Sciences Research Council (grant EP/L016508/01), the Scottish Funding Council, Heriot-Watt
University and the University of Edinburgh.
J.\,F. also acknowledges support from Tadahiro Oh's ERC starting grant no. 637995 ProbDynDispEq.
L.T.~was supported by the Deutsche
Forschungsgemeinschaft (DFG, German Research Foundation) under Germany's Excellence
Strategy-EXC-2047/1-390685813, through the Collaborative Research Centre (CRC) 1060. 
J.F. and L.T. would like to thank Tadahiro Oh for his continuous support during the preparation of this paper.
\end{ackno}


\begin{thebibliography}{99}




\bibitem{BdP}
V. Barbu, G. Da Prato, \emph{The stochastic nonlinear damped wave equation},
Appl. Math. Optim. 46 (2002), no. 2-3, 125--141. 

\bibitem{Bourgain2}
 J.~Bourgain, 
 {\it Periodic nonlinear Schr\"{o}dinger equation and invariant measures,}
  Comm. Math. Phys. 166 (1994), no. 1, 1--26.

  
\bibitem{bos16}
Z. Brze\'zniak, M. Ondrej\'at, J. Seidler,
\emph{Invariant measures for stochastic nonlinear beam and wave equations},
J. Differential Equations 260 (2016), no. 5, 4157--4179.

 
 \bibitem{BTglobal}
 N.~Burq, N.~Tzvetkov, 
 {\it Random data Cauchy theory for supercritical wave equations. II. A global existence
result,} Invent. Math. 173 (2008), no. 3, 477--496.

\bibitem{BW}
O.~Butkovsky, F.~Wunderlich, \emph{Asymptotic strong Feller property and local weak irreducibility via generalized couplings},  \texttt{arXiv:1912.06121}, [math.PR].

\bibitem{cmw}
A.~Chandra, A.~Moinat, H.~Weber, \emph{A priori bounds for the $\Phi^4$ equation in the full sub-critical regime}, \texttt{arXiv:1910.13854} [math.AP].
 
 \bibitem{DMVT}
  J.~Colliander, M.~Keel, G.~Staffilani, H.~Takaoka, T.~Tao, {\it A refined global well-posedness result for
Schr\"odinger equations with derivative}, SIAM J. Math. Anal. 34 (2002), no. 1, 64--86.


\bibitem{dpd}
 G.~Da Prato, A.~Debussche, 
 {\it Two-dimensional Navier-Stokes equations driven by a space-time white
noise}, J. Funct. Anal. 196 (2002), no. 1, 180--210.

\bibitem{DPZ1}
G.~Da Prato, J.~Zabczyk, 
{\it Stochastic equations in infinite dimensions},
Second edition. Encyclopedia of Mathematics and its Applications, 152. Cambridge University Press, Cambridge, 2014. xviii+493 pp. ISBN: 978-1-107-05584-1.


\bibitem{gh1}
M.~Gubinelli, M.~Hofmanov\'a, \emph{Global solutions to elliptic and parabolic $\Phi^4$ models in Euclidean space}, Comm. Math. Phys. 368 (2019), no. 3, 1201--1266.


\bibitem{gh2}
M.~Gubinelli, M.~Hofmanov\'a, \emph{A PDE construction of the Euclidean $\Phi^4_3$ quantum field theory}, \texttt{arXiv:1810.01700} [math-ph].


\bibitem{GKO}
 M.~Gubinelli, H.~Koch, T.~Oh, {\it Renormalization of the two-dimensional stochastic nonlinear wave equations,}
 Trans. Amer. Math. Soc. 370 (2018), no 10, 7335--7359.

\bibitem{GKOT}
 M.~Gubinelli, H.~Koch, T.~Oh, L.~Tolomeo,
  {\it Global dynamics for the two-dimensional stochastic nonlinear wave equations,}
 \texttt{arXiv:2005.10570}  [math.AP].
 

\bibitem{HM1}
M.~Hairer, J.~Mattingly,
{\it Ergodicity of the 2D Navier-Stokes equations with degenerate stochastic forcing,}
Ann. of Math. (2) 164 (2006), no. 3, 993--1032.


\bibitem{HM2}
M.~Hairer, J.~Mattingly,
{\it The strong Feller property for singular stochastic PDEs,}
Ann. Inst. Henri Poincaré Probab. Stat. 54 (2018), no. 3, 1314--1340.

\bibitem{HM3}
M.~Hairer, J.C.~Mattingly, M.~Scheutzow, \emph{Asymptotic coupling and a general form of Harris theorem with applications to stochastic delay equations}. Probab. Theory Relat. Fields 149, 223--259 (2011).


  \bibitem{mckean}
 H.\,P.~McKean, 
 {\it Statistical mechanics of nonlinear wave equations. IV. Cubic Schr\"odinger}, Comm. Math.
Phys. 168 (1995), no. 3, 479--491. {\it Erratum: Statistical mechanics of nonlinear wave equations. IV. Cubic
Schr\"odinger,} Comm. Math. Phys. 173 (1995), no. 3, 675.

\bibitem{mw1}
J.-C.~Mourrat, H.~Weber, \emph{Global well-posedness of the dynamic $\Phi^4$ model in the plane}, Ann. Probab. 45 (2017), no. 4, 2398--2476.

\bibitem{mw2}
J.-C.~Mourrat, H.~Weber, \emph{The dynamic $\Phi^4_3$ model comes down from infinity}, Comm. Math. Phys. 356 (2017), no. 3, 673--753. 
  

  \bibitem{BOP3}
T.~Oh, O.~Pocovnicu,
  {\it Probabilistic global well-posedness of the energy-critical defocusing quintic non-linear wave equation on $\R^3$,} 	
 J. Math. Pures Appl. 105 (2016), 342--366.
  

\bibitem{PaWu}
G. Parisi, Y. S. Wu, \emph{Perturbation theory without gauge fixing}, Scientia Sinica, Zhongguo Kexue, 24(4):483--496, 1981.

\bibitem{RevuzYor}
D.~Revuz, M.~Yor, 
{\it Continuous martingales and Brownian motion.
Third edition}, Grundlehren der Mathematischen Wissenschaften [Fundamental Principles of Mathematical Sciences], 293. Springer-Verlag, Berlin, 1999. xiv+602 pp. ISBN: 3-540-64325-7.
	

\bibitem{t18wave}
L.~Tolomeo, 
{\it Global well-posedness of the two-dimensional stochastic nonlinear wave equation on an unbounded domain}, 
to appear in Ann. Probab. 
 
\bibitem{t18erg}
L.~Tolomeo, 
{\it 
Unique ergodicity for a class of stochastic hyperbolic equations with additive space-time white noise}, 
 Comm. Math. Phys. 377 (2020), no. 2, 1311--1347.

\bibitem{t19thesis}
L.~Tolomeo,
\emph{Stochastic dispersive PDEs with additive space-time white noise,} 
PhD Thesis, University of Edinburgh, 2019.

\bibitem{Tolomeo3}
L.~Tolomeo, 
{\it Ergodicity for the hyperbolic $P(\Phi)_2$-model}, 
in preparation.

\bibitem{TW}
P.~Tsatsoulis, H.~Weber, 
{\it Spectral gap for the stochastic quantization equation on the 2-dimensional torus},
Ann. Inst. Henri Poincar\'e Probab. Stat. 54 (2018), no. 3, 1204--1249.



\end{thebibliography}
\end{document}